\documentclass[12pt,oneside]{amsart}
\usepackage{amssymb, amsmath, amsthm}

\usepackage{epsfig}
\usepackage{graphicx}
\usepackage{color}
\usepackage{enumerate}
\usepackage{amsrefs}
\usepackage[pagewise]{lineno}

\usepackage{pinlabel}

\theoremstyle{plain}
\newtheorem{theorem}{Theorem}[section]
\newtheorem{question}[theorem]{Question}

\newtheorem*{theorem*}{Theorem}
\newtheorem*{theorem-DisjtSS}{Theorem \ref{Thm: Disjt SS}}
\newtheorem*{theorem-EssentialTorus}{Theorem \ref{Thm: Essential Torus}}
\newtheorem*{cor-ScharlWu}{Corollary \ref{Cor-ScharlWu}}
\newtheorem*{corollary-MSC 2}{Corollary \ref{Cor: MSC 2}}
\newtheorem*{theorem-Main A}{Theorem \ref{Thm: Main A}}
\newtheorem*{theorem-Main B}{Theorem \ref{Thm: Main Thm B}}
\newtheorem*{Cor-Unknotting}{Theorem \ref{Cor: Prime Unknotting 1}}
\newtheorem*{Cor-genusbandsum}{Theorem \ref{Thm: Genus superadd}}
\newtheorem*{Cor-bandsumscc}{Corollary \ref{Cor: Band Sums CC}}
\newtheorem*{maintheorem-intro}{Theorem \ref{Main Theorem}}

\newtheorem{corollary}[theorem]{Corollary}
\newtheorem{lemma}[theorem]{Lemma}

\theoremstyle{definition}
\newtheorem{remark}[theorem]{Remark}

\theoremstyle{definition}

\newcommand{\up}{\uparrow}
\newcommand{\dn}{\downarrow}

\newcommand{\defn}[1]{\emph{#1}}
\newcommand{\bdd}{\partial}
\newcommand{\boundary}{\partial}

\newcommand{\mc}[1]{\mathcal{#1}}

\begin{document}

\title{Distance two links}
\author{Ryan Blair, Marion Campisi, Jesse Johnson, Scott A. Taylor, Maggy Tomova}

\begin{abstract} In this paper, we characterize all links in $S^3$ with bridge number at least three that have a bridge sphere of distance two. We show that a link $L$ has a bridge sphere of distance at most two then it falls into at least one of three categories:
\begin{itemize}
\item The exterior of $L$ contains an essential meridional sphere.

\item $L$ can be decomposed as a tangle product of a Montesinos tangle with an essential tangle in a way that respects the bridge surface and either the Montesinos tangle is rational or the essential tangle contains an incompressible, boundary-incompressible annulus.

\item $L$ is obtained by banding from another link $L'$ that has a bridge sphere of the same Euler characteristic as the bridge sphere for $L$ but of distance 0 or 1.
\end{itemize}
\end{abstract}

\thanks{This paper was written as part of an AIM SQuaREs project. The third author was also supported by NSF grant DMS-1006369 and the fifth author was supported by NSF grant DMS-1054450. The fourth author was partially supported by a grant from the Natural Sciences Division of Colby College}

\maketitle

\section{Introduction}

Following Hempel's definition of distance for Heegaard splittings~\cite{He}, there have been a number of results showing that knots with high distance bridge surfaces are well behaved. For example, high distance knots are hyperbolic and do not contain essential surfaces of small euler characteristic~\cite{BS} and they have unique minimal bridge surfaces~\cite{To07}. Additionally, connect sums of knots with high distance bridge surfaces have multiple distinct bridge surfaces~\cite{JT}. Most recently, the Cabling Conjecture was established for all knots with a bridge surface of distance at least three~\cite{BCJTT, BCJTT2}.

These results motivate us to study low distance knots and links. (We will use the term ``link'' below to mean both one-component and multiple-component links.) Links that have a bridge surface of distance two are of particular interest because a link with a minimal bridge sphere of distance one always contains an essential meridional surface~\cite{Th} and links with bridge surfaces of distance greater than two share many nice properties that often allow them to be addressed collectively as ``high distance" links.

In this paper, we prove the following theorem. Precise definitions will be given later.

\begin{theorem}\label{thm:main}
Let $L$ be a link in $S^3$ and let $\Sigma$ be a bridge sphere for $L$ of distance 2. Then at least one of the following holds:
\begin{enumerate}
\item There is an essential meridional sphere in the exterior of $L$,

\item $L$ can be decomposed as a tangle product of a bridge-split Montesinos tangle $(B_1,T_1)$ with an essential tangle $(B_2,T_2)$ such that $\Sigma \cap B_1$ is a disk bridge surface for $(B_1,T_1)$ and either $(B_1,T_1)$ is rational or the exterior of $T_2$ in $B_2$ contains an incompressible, boundary incompressible annulus, or

\item $L$ can be obtained via banding from a link $L'$ that has a bridge sphere with the same Euler characteristic as $\Sigma \setminus \eta(L)$ but has distance 0 or 1.
\end{enumerate}
\end{theorem}

Conclusion (2) of Theorem \ref{thm:main} is illustrated in Figure \ref{fig:ThmExMont}. Conclusion (3) is illustrated in Figure \ref{fig:ThmExBand}. All of the conclusions are discussed in detail in Section~\ref{sect:constr}.

\begin{figure}[htb]
  \begin{center}
  \includegraphics[width=2in]{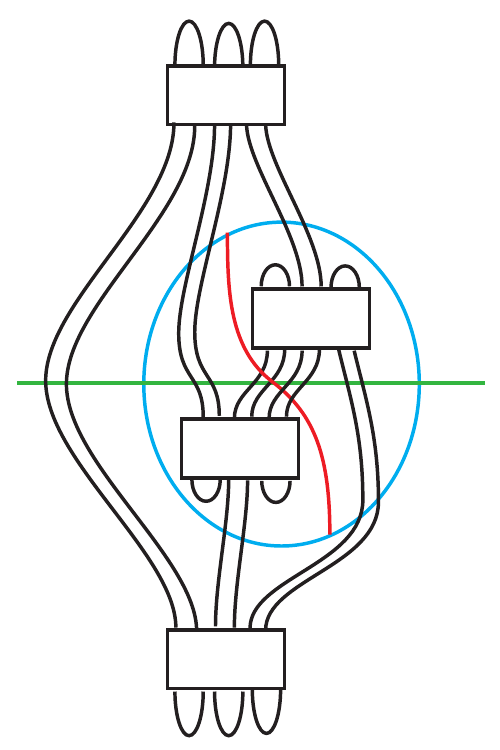}
  \caption{An example of a link decomposed as the tangle product of a bridge-split Montesinos tangle $(B_1,T_1)$ with a tangle $(B_2,T_2)$ such that the bridge surface for $L$ is a disk bridge surface for $(B_1,T_1)$. (Each group of strands may, in general, consist of arbitrarily many strands.)}
  \label{fig:ThmExMont}
  \end{center}
\end{figure}

\begin{figure}[htb]
  \begin{center}
  \includegraphics[width=2in]{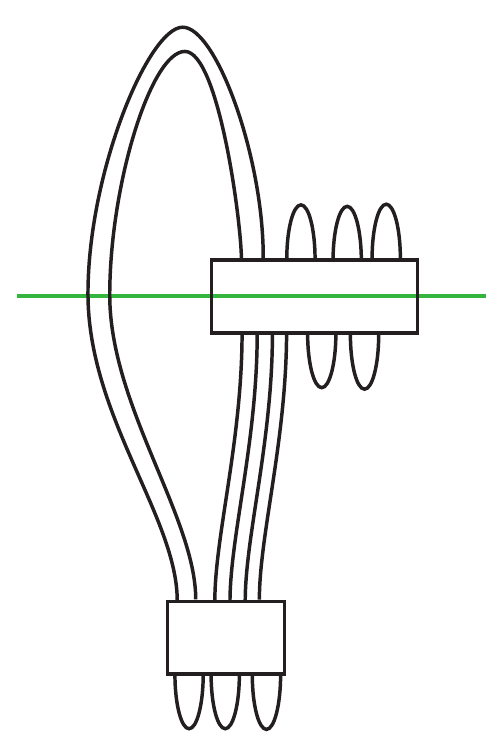}
  \caption{An example of a link obtained by banding.}
  \label{fig:ThmExBand}
  \end{center}
\end{figure}

We note that this paper is the analogue to recent results about 3-manifolds with distance-two Heegaard splittings. Hempel~\cite{He} and Thompson~\cite{Th3} (independently) classified all 3-manifolds that have genus two, distance-two Heegaard splittings. In work in preparation, Rubinstein and Thompson~\cite{RT} have proved a similar classification theorem for all manifolds that have a distance-two Heegaard splitting of any genus. Portions of the proof of Theorem \ref{thm:main} are inspired by the work of Rubinstein and Thompson on distance-two Heegaard splittings.

The analogy between the present work and these results about Heegaard splittings is most easily seen in the three-bridge case. Theorem~\ref{thm:main} implies the following:

\begin{corollary}
\label{coro:main}
If $L$ is a 3-bridge knot with a minimal bridge sphere of distance two, then one of the following holds:

\begin{enumerate}
\item The exterior of $L$ contains an incompressible meridional 4-punctured sphere.
\item $L$ is obtained by banding a 3-bridge presentation of either the unknot, a 2-bridge knot, or the connected sum of two 2-bridge knots.
\item $L$ is a small Montesinos knot.
\end{enumerate}
\end{corollary}

Recall that every genus-two 3-manifold is the double-branched cover of $S^3$ over a three-bridge knot and that every genus-two Heegaard surface is the lift of a 6-punctured bridge sphere to the double-branched cover. Additionally, every essential loop in a genus-two Heegaard surface is the lift of an essential loop in the corresponding bridge sphere. This fact implies that the distance of a 6-punctured bridge sphere is equal to the distance of the corresponding genus-two Heegaard surface in the double-branched cover. Thus, as expected, the Heegaard splittings in the Hempel/Thompson classification are precisely the double branched covers of the links in Corollary~\ref{coro:main}.

 We describe the types of distance-two bridge spheres in Section~\ref{sect:constr}. A number of technical lemmas are proved in Section~\ref{sect:simplifying}, followed by the proof of Theorem~\ref{thm:main} and Corollary~\ref{coro:main}  in Section~\ref{sect:proof}.

\section{Definitions and Constructions}
\label{sect:constr}

\subsection{Bridge surfaces} Let $L$ be a link in the 3-sphere and $\eta(L)$ a regular neighborhood of $L$. Let $h$ be the standard height function from $S^3$ to $[0,1]$ whose level surfaces are concentric 2-spheres. A level sphere $\Sigma = h^{-1}(t)$ is called a \emph{bridge sphere} for $L$ if all maxima of $h|_L$ are above $\Sigma$ and all minima of $h|_L$ are below $\Sigma$. In particular, if $\Sigma$ is a bridge sphere for $L$, then $L$ intersects each of the 3-balls bounded by $\Sigma$ in boundary parallel arcs called \emph{bridges}. The disks of parallelism are called \emph{bridge disks}. More precisely, a bridge disk is a disk with interior disjoint from $L$ and the bridge surface, and whose boundary is the union of an arc in $L$ and an arc in the bridge surface. For technical reasons, we will only discuss links whose bridge spheres have bridge number at least three (so at least 6 punctures). However, two-bridge knots and links are well understood~\cite{HT, STo} and every one-bridge knot is an unknot.

Given a Morse function $h: M \rightarrow [0,1]$ where $M$ is any 3-manifold, an arc or a surface $F$ properly embedded in $M$ is \emph{vertical} if $h|_F$ does not have any critical points in the interior of $F$.

\subsection{Distance}

The \emph{curve complex} of a compact surface $F$ is the simplicial complex with vertices that correspond to isotopy classes of essential (including not boundary parallel) simple closed curves in $F$. An edge connects two vertices if the corresponding isotopy classes of curves have disjoint representatives. The \emph{distance} of a bridge sphere $\Sigma$ for a link $L$ in $S^3$ is the length of the shortest path in the curve complex for $\Sigma \setminus \eta(L)$ from the boundary of a compressing disk for $\Sigma \setminus \eta(L)$ in $S^3 \setminus \eta(L)$ above $\Sigma$ to the boundary of a compressing disk for  $\Sigma \setminus \eta(L)$ in $S^3 \setminus \eta(L)$ below $\Sigma$. The notion of distance measured in the curve complex was first introduced by Hempel in \cite{He}. There is also a similar notion of distance measured in the arc and curve complex introduced by Bachman and Schleimer in \cite{BS}.

For a bridge sphere of distance two, every compressing disk on one side of the bridge sphere intersects every compressing disk on the other side non-trivially, but there is an essential simple closed curve $\alpha$ in $\Sigma \setminus \eta(L)$ that is disjoint from at least one compressing disk on each side of $\Sigma \setminus \eta(L)$. For each of the possible conclusions of our main theorem, we consider the possibility of a converse.

\subsection{Incompressible spheres}

Recall that a \textit{meridional surface} for a link $L \subset S^3$ is a surface properly embedded in $S^3 \setminus \eta(L)$ whose boundary is a set of meridional curves in $\bdd \eta (L)$. A meridional surface $C$ is \textit{compressible} if there is a disk in $S^3 \setminus\eta(L)$ with interior disjoint from $C$ whose boundary is contained in $C$ and is essential in $C \setminus \eta(L)$. A meridional surface $C$ is \textit{incompressible} if it is not compressible.

The first possible conclusion of our Main Theorem is that our link has an incompressible, boundary-incompressible planar meridional surface. In general, it is likely that there are knots of distance at least 3 which have essential meridional planar surfaces in their exterior. However, if there is an essential planar surface C for the link which intersects the bridge surface $\Sigma$ in a single simple closed curve which is essential on both surfaces, such that $C$ is disjoint from compressing discs for $\Sigma$ on both sides of $\Sigma$, then the link will have distance at most 2. The intersection curve $C \cap \Sigma$ is the middle vertex in a length two path in the curve complex of $\Sigma$. In some cases, the planar surface we construct in the proof of the Main Theorem will be of this type.

\subsection{Disk Bridge Surfaces and Montesinos Tangles}

In what follows, it will be natural to study tangles from the perspective of the height function on $D^2\times [-1,1]$. Let $\phi$ be the projection map $D^2\times [-1,1] \rightarrow [-1,1]$. Given a tangle $T$ properly embedded in $D^2 \times [-1,1]$ such that $\partial T \subset D^2\times \{-1,1\}$, we will say that $T$ is in \emph{bridge position} if all of the maxima of $\phi|_T$ project into $(0,1)$ and all of the minima of $\phi|_T$ project into $(-1,0)$. A disk $\phi^{-1}(s)$ is a \emph{disk bridge surface} for $T$ if $s$ is a regular value of the height function $\phi|_T$, and all the minima of $T$ lie below and all the maxima lie above it.

Recall that an $n$-strand tangle $(B,T)$ is \emph{rational} if each of the $n$ strands can be simultaneously isotoped into $\partial B$ while fixing their endpoints. An $n$-strand tangle $(B,T)$ is \emph{essential} if $\partial B$ is incompressible when viewed as a meridional, planar surface. Let $(B_i, T_i)$ for $i = 1,2$ be a rational tangle and let $D_i \subset \boundary B_i$ be an incompressible $k$-punctured disc for some $k \geq 0$. The result $(B,T)$ of gluing $(B_1, T_1)$ and $(B_2, T_2)$ along $D_1$ and $D_2$ is a \defn{Montesinos tangle} if the image of $D_1$ and  $D_2$ is not boundary-parallel in $(B,T)$. The disc $D = D_1 = D_2$ in $(B,T)$ is a \defn{Montesinos disk} for $(B,T)$. Given a Montesinos tangle $T$ properly embedded in $D^2 \times [-1,1]$ such that $\partial T \subset D^2\times \{-1,1\}$, we say that $T$ is \emph{bridge-split} if $\phi|_T$ has both maxima and minima, the foliation induced by $\phi$ on $D$ consists of parallel arcs and $D$ separates the maxima of $T$ from the minima of $T$, as in Figure \ref{fig:BridgeSplit}. A knot $K$ in $S^3$ is a \emph{tangle product} of tangles $(B_1, T_1)$ and $(B_2, T_2)$ if $S^3$ is the result of gluing $B_1$ to $B_2$ along their boundary and $K=T_1\cup T_2$.

Suppose a link $K$ has a tangle decomposition such that one of the tangles $(B,T)$ is a bridge-split Montesinos tangle. If the bridge surface $\Sigma$ meets $B$ in a disk bridge surface for $T$, then $\Sigma$ has distance at most two. A disk bridge surface for a bridge-split Montesinos tangle is depicted in green in Figure \ref{fig:BridgeSplit}.

\begin{figure}[htb]
  \begin{center}
  \includegraphics[width=1.5in]{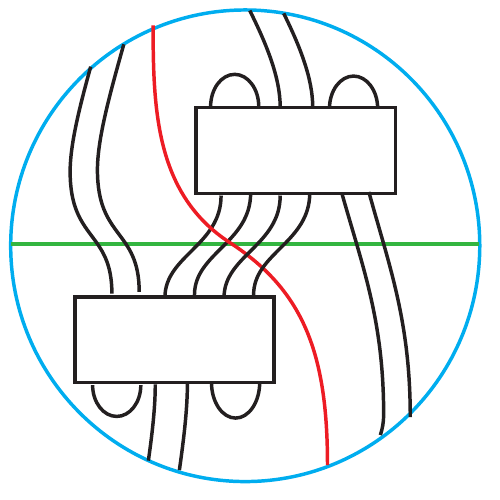}
  \caption{A bridge-split Montesinos tangle}
  \label{fig:BridgeSplit}
  \end{center}
\end{figure}

\subsection{Banding}
\label{sec:banding}
A third way to construct a bridge surface of distance at most two is as follows: Let $L'$ be any link such that some local maximum $m_1 \in L'$ lies below some local minimum $m_2 \in L'$ with respect to the standard height function $h$. Notice that if we monotonically isotope all maxima of $L'$ above all minima of $L'$, there will be a compressing disk for one of the resulting bridge spheres $\Sigma$ that surrounds $m_1$ and a second disjoint compressing disk around $m_2$, on the other side of $\Sigma$. Thus, the level sphere corresponding to a bridge sphere for $L'$ will have distance at most 1. We can also construct a distance-one bridge surface for any link by perturbing a given bridge surface, i.e.\ adding a new pair of canceling bridges. This new bridge surface will no longer be minimal.

Choose any vertical arc $\beta$ with one endpoint at $m_1$ and the other endpoint at a local maximum above $m_2$, as in Figure \ref{fig:banding}. Let $b \subset S_3$ be a band along $\beta$, i.e.\ the image of a square $[0,1] \times [0,1]$ such
\begin{itemize}
\item the interior of $b$ is disjoint from $L'$,
\item $b$ is vertical with respect to $h$,
\item $\beta$ is the image of $\frac{1}{2} \times [0,1]$ and
\item the arcs $[0,1] \times \{0\}$ and $[0,1] \times \{1\}$ are contained in small neighborhoods of $b \cap L'$.
\end{itemize}
Let $L$ be the result of replacing the arcs $[0,1] \times \{0\}$ and $[0,1] \times \{1\}$ of $L'$ with the arcs $\{0\} \times [0,1]$ and $\{1\} \times [0,1]$. Note that even after we choose $\beta$, there are infinitely many different possible bands $b$, related by twisting around $\beta$. In particular, if $L'$ is a knot, but our initial choice of $b$ produces a link, we can add a half twist to ensure that $L$ is a knot. We will say that $L$ is obtained from $L'$ by \emph{banding}.

More generally, the banding operation described here is also known as \emph{integer-sloped rational tangle replacement}. However, for our purposes it suffices to note that the resulting link has a bridge surface of distance at most two. To see this, let $S$ be the level sphere immediately above the highest minimum of $L'$ and let $\ell$ be the boundary of a regular neighborhood in $S$ of $b \cap S$. Any maximum of $L'$ disjoint from $b$ gives rise to a compressing disk for $S$ contained above $S$ and disjoint from $\ell$. Similarly, any minimum of $L'$ that does not cobound (with $m_1$) a monotone subarc of $L'$ will give rise to a compressing disk for $S$ contained below $S$ and disjoint from $\ell$.  Thus, $S$ has distance at most two.
 By choosing a ``sufficiently complicated'' arc $\beta$, we expect that one can always construct a link of distance exactly two, but proving this is beyond the scope of the present paper.
\begin{figure}[htb]
  \begin{center}
  \includegraphics[scale=0.4]{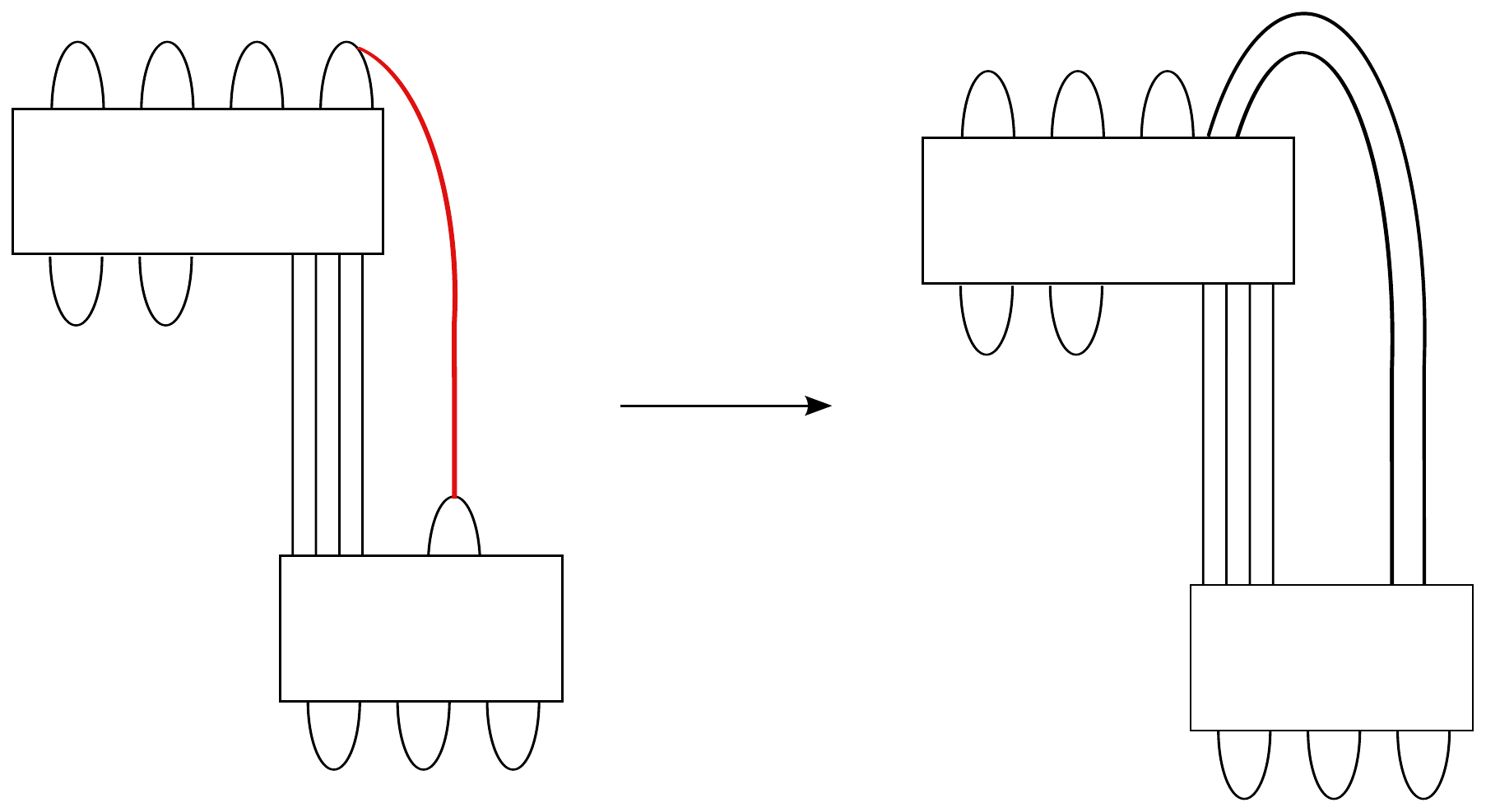}
    \put(-175,90){$\beta$}
  \caption{}
  \label{fig:banding}
  \end{center}
\end{figure}

\section{Simplifying bridge disks}
\label{sect:simplifying}

In this section, we introduce several results that will be needed in the proof of the main theorem. We will have several occasions to consider a tangle embedded in a cylinder and a disjoint collection of boundary compressing disks for its disk bridge surface. In the next lemma, we show that we can always assume that such boundary compressing disks for the disk bridge surface are vertical with respect to the natural height function on $D^2\times [-1,1]$ given by projection onto the second component. We can in fact assume so while also assuming that the tangle is polygonal and composed only of vertical and horizontal components.

\begin{lemma}
\label{lem:vertical boundary compressing disks}
Let $T$ be a tangle properly embedded in a cylinder $V=D^2 \times [-1,1]$ such that all endpoints of $T$ are in $D^2 \times \{-1,1\}$. Suppose that $D^2 \times \{0\}$ is a disk bridge surface for $T$. Let $\mc{D}^{\up}$ and $\mc{D}^{\dn}$ be collections of pairwise disjoint boundary compressing disks for $D^2 \times \{0\}$ above and below $D^2 \times \{0\}$ respectively such that $\mc{D}^{\up} \cap \mc{D}^{\dn}=\emptyset$. Then there is an isotopy during which the tangle remains transverse to $D^2 \times \{0\}$ and after which all disks in $\mc{D}^{\up}$ and $\mc{D}^{\dn}$ are disjoint and vertical. Furthermore, we may assume that after the isotopy, the tangle $T$ is a set of polygonal curves composed only of horizontal and vertical arcs such that all horizontal arcs are contained in $D^2\times \{\frac{1}{2}\}$ or $D^2\times \{-\frac{1}{2}\}$ and each component of $T\setminus (D^2 \times \{0\})$ contains at most one horizontal arc.
\end{lemma}

\begin{proof}
Let $B^{\up}=D^2 \times [0,1]$ and $B^{\dn}=D^2 \times [-1,0]$. The proof follows the following basic structure. We begin by using standard techniques to construct an ambient isotopy so that $D^{\dn}$, $T\cap B^{\dn}$ and subsequently $T\cap B^{\up}$ have the desired properties. We then analyze the intersection between $D^{\up}$ and the vertical bridge disks for $T\cap B^{\up}$ to construct an ambient isotopy that results in $\mc{D}^{\up}$ being vertical and disjoint from $\mc{D}^{\dn}$ while preserving the desired properties of $D^{\dn}$, $T\cap B^{\dn}$ and $T\cap B^{\up}$.

First, perform an isotopy supported in a neighborhood of $B^{\dn}$ after which each arc of $T\cap B^{\dn}$ that is parallel into $D^2\times \{0\}$ consists of two vertical arcs and one horizontal arc in $D^2\times \{-\frac{1}{2}\}$ and all other arcs of $T\cap B^{\dn}$ are $I$ fibers of $D^2 \times [-1,0]$. Because the boundary of each disk of $\mc{D}^{\dn}$ consists of two vertical and two horizontal arcs, each disk is isotopic to a vertical disk. By the Isotopy Extension Theorem, these isotopies extend to a single ambient isotopy in a neighborhood of $B^{\dn}$ that fixes $T\cap B^{\dn}$ and results in $\mc{D}^{\dn}$ being vertical. Note that this ambient isotopy preserves $\mc{D}^{\up} \cap \mc{D}^{\dn}=\emptyset$.

Next, perform an isotopy supported in $V \setminus \eta(B^{\dn})$ (where $\eta(B^{\dn}
)$ is a regular neighborhood of $B^{\dn}$) after which each arc of $T\cap B^{\up}$ that is parallel into $D^2\times \{0\}$ consists of two vertical arcs and one horizontal arc in $D^2\times \{\frac{1}{2}\}$ and every other arc of $T \cap B^{\up}$ is an $I$ fiber of $D^2 \times [0,1]$. This isotopy does not affect $\mc{D}^{\dn}$ or $T \cap B^{\dn}$. After these isotopies, $T$ meets the conclusion of the theorem, with the exception that it remains to be shown that we can now perform an isotopy that guarantees that $\mc{D}^{\up}$ is a collection of vertical disks without affecting the conditions on $T$ and $\mc{D}^{\dn}$ that we have already imposed.

Let $\Delta$ be a complete collection of vertical bridge disks for $T \cap B^{\up}$. Consider $\Delta \cap \mc{D}^{\up}$. After an isotopy of $\mc{D}^{\up}$ supported in the interior of $B^{\up}$ which does not affect the tangle or $\Delta$, we may assume that this intersection consists only of arcs. Let $\alpha$ be an outermost arc of intersection on some disk $D \in \mc{D}^{\up}$. This arc cuts off a subdisk $E \subset D$ whose interior is disjoint from all disks in $\Delta$ and a subdisk $F$ from some disk in $\Delta$. Let $B$ be the ball cobounded by $E$, $F$ and a subdisk $G$ of $D^2 \times \{0\}$.

By standard transversality arguments, we can assume that there is $\epsilon >0$ so that $E \cap (D^2 \times [0, \epsilon])$ is vertical.  Let $\{\kappa_i\}$ be the set of all bridges contained in $B$. Since $\Delta$ is disjoint from $E$, the bridge disk of each $\kappa_i$ is disjoint from $E$ so there is a vertical isotopy of $\{\kappa_i\}$ relative to $\Delta$ and supported in the interior of $B$ after which arcs in $\{\kappa_i\}$ are contained in $D^2 \times [0, \epsilon]$. Now the interior of $B \cap (D^2 \times [\epsilon, 1])$ is a ball that is completely disjoint from the link and thus there is an isotopy shrinking $B \cap (D^2 \times [\epsilon, 1])$ so it lies in a neighborhood of $G \cup F$, as in Figure \ref{fig:isotopy}.
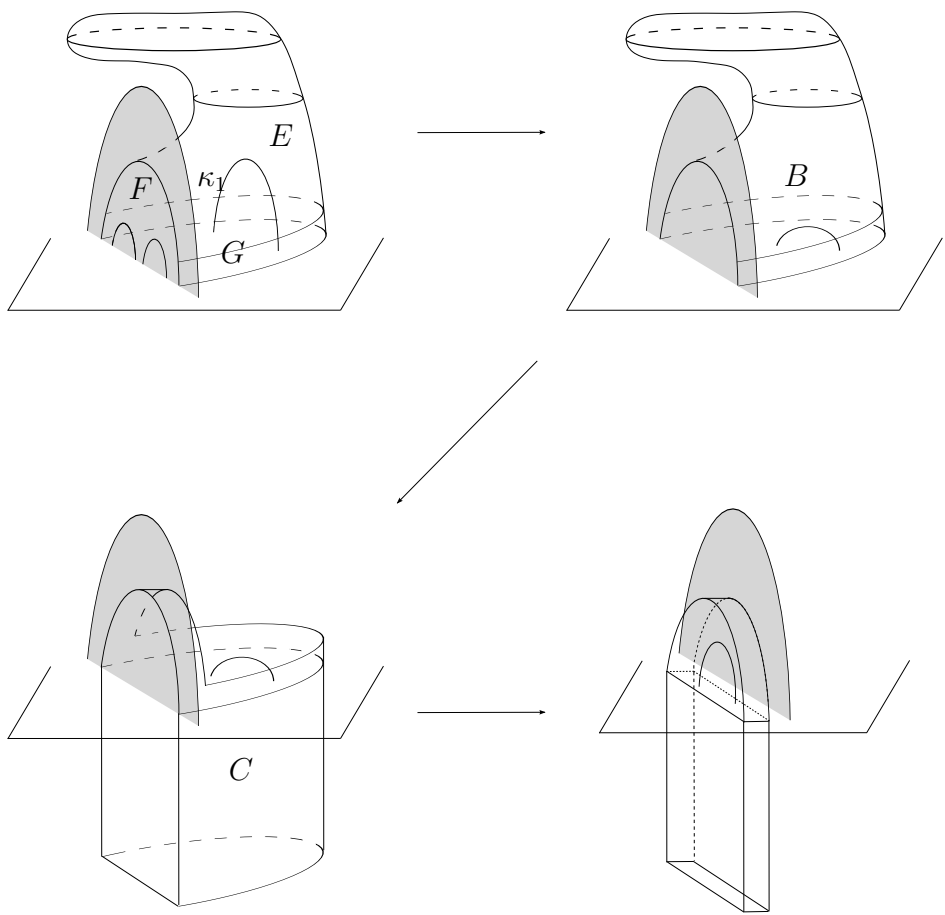
\begin{figure}[htb]
  \begin{center}
  \includegraphics[scale=0.25]{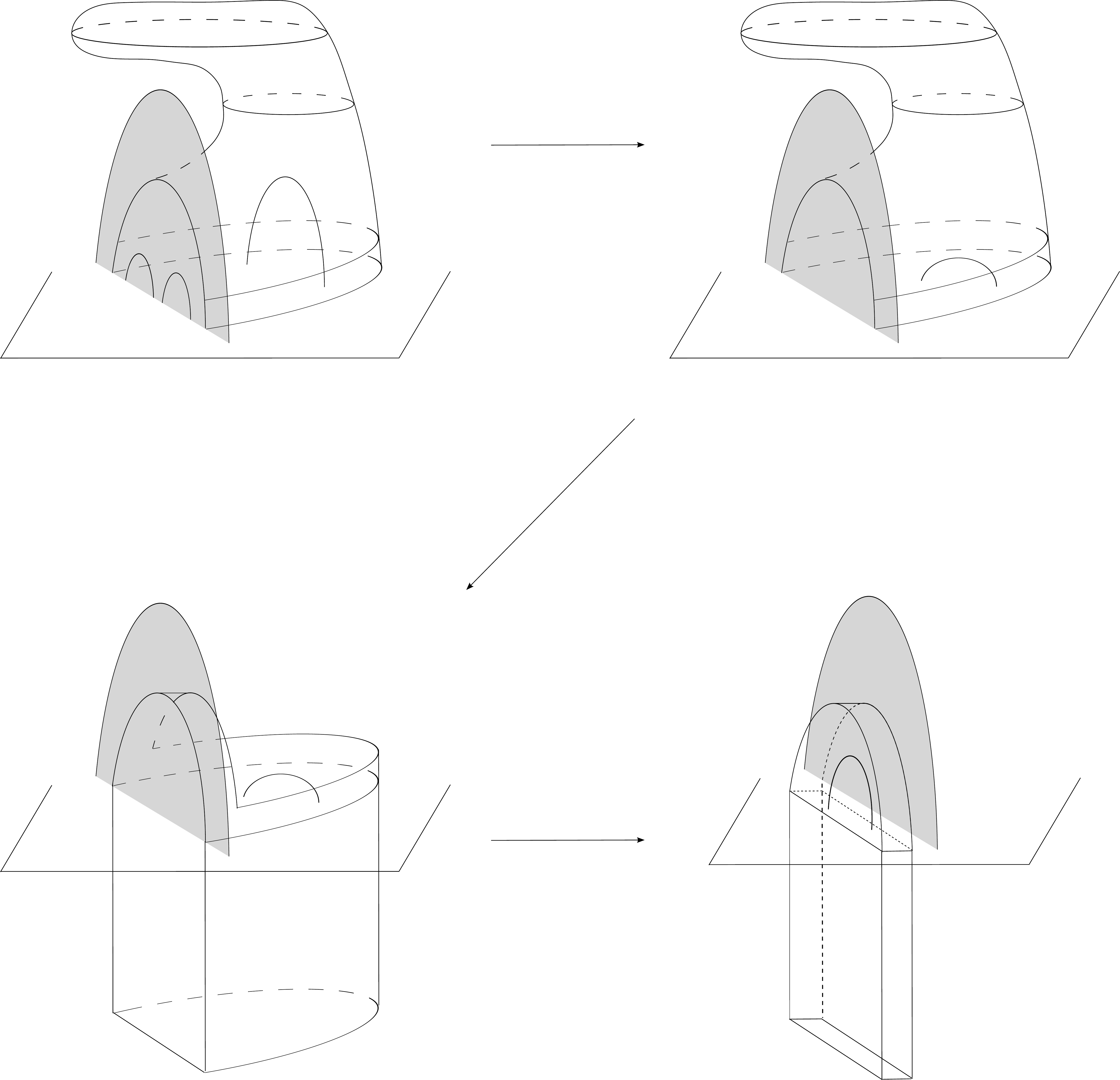}
  \put(-255,290){$E$}
  \put(-308,270){$F$}
  \put(-273,246){$G$}
  \put(-282,275){$\kappa_1$}
  \put(-60,275){$B$}
  \put(-270,50){$C$}
  \caption{In the first figure we have included some arcs of intersection between $\Delta$ and $\mc{D}^{\up}$ to remind the reader that $\alpha$ is not necessarily outermost on $\Delta$. We have dropped these arcs in the rest of the figures to simplify the picture.}
  \label{fig:isotopy}
  \end{center}
\end{figure}

Let $C$ be the cylinder $G \times [-1, 0]$ together with the ball $B$ after the above isotopies. The cylinder $G \times [-1, 0]$ may, of course, intersect $T$ both in its boundary and in its interior. Perform a horizontal isotopy of $C$, first shrinking it horizontally until $C \cap B^{\up}$ is contained in a neighborhood of $F$ and then pushing $C$ through $F$, thus reducing $\Delta \cap \mc{D}^{\up}$, as in Figure \ref{fig:isotopy}. This sequence of isotopies preserves the property that $T$ is composed of horizontal and vertical arcs and that $\mc{D}^{\dn}$ is vertical. Note also that these ambient isotopies preserve the property that $\mc{D}^{\up} \cap \mc{D}^{\dn}=\emptyset$.

Thus, we may assume that $\Delta \cap \mc{D}^{\up} = \emptyset$. However, if the collection of disks $\mc{D}^{\up}$ is contained in the complement of all vertical bridge disks then they are in fact contained in a product region and thus can be isotoped to be vertical via an isotopy that is supported in $D^2 \times [\epsilon, 1]$, does not affect $T$, and preserves the fact that $\mc{D}^{\up} \cap \mc{D}^{\dn}=\emptyset$.
\end{proof}

The following is a standard sweepout argument similar to that of \cite{G}.

\begin{lemma}
\label{lem:weakly reducible}
Let $(B,T)$ be a tangle with $B=(D^2\times [-1,1])$ such that $\partial T \subset (D^2\times \{-1,1\})$ and let $\Sigma$ be a bicompressible disk bridge surface for $T$. Let $\phi: D^2 \times [-1,1] \rightarrow [-1,1]$ be the natural projection map. Suppose $D$ is a compressing disk for $\bdd B$ in $B\setminus T$ such that $\partial D$ cannot be isotoped to be disjoint from $\partial \Sigma$. Then there is an isotopy of $D$ that fixes $\partial D$ and the height and number of saddles of $D$ after which one of the following occurs:
\begin{enumerate}
\item There exists a value $t$ of $\phi|_{D}$ such that $\phi^{-1}(t)$ is isotopic to $\Sigma$ and there exists a compressing disk for $\phi^{-1}(t)$ that is disjoint from  a compressing disk or boundary compressing disk for $\phi^{-1}(t)$ on the opposite side of $\phi^{-1}(t)$,

\item there exists a regular value $t$ of $\phi|_{D}$ such that $\phi^{-1}(t)$ is isotopic to $\Sigma$ and there are two disjoint boundary compressing disks for $\phi^{-1}(t)$ with boundary contained in $\phi^{-1}(t)\cap D$ and on opposite sides of $\phi^{-1}(t)$, or

\item there exists a critical value $t$ of $\phi|_{D}$ such that $\phi^{-1}(t)$ is isotopic to $\Sigma$ and there are two disjoint boundary compressing disks for $\phi^{-1}(t)$ on opposite sides of $\phi^{-1}(t)$ with the property that the boundaries of these disks meet $\phi^{-1}(t)$ in the boundary of a regular neighborhood of the component of $\phi^{-1}(t)\cap D$ that contains a saddle.
\end{enumerate}
\end{lemma}

\begin{proof}
Consider the projection map $\phi: D^2 \times [-1,1] \rightarrow [-1,1]$. We have $\phi^{-1}(t) = D^2 \times \{t\}$ for each $t \in [-1,1]$ and we will assume $\Sigma = \phi^{-1}(0)$. Isotope all maxima of $T$ up above all critical values of $\phi|_D$ into a neighborhood of $D\times \{1\}$ and isotope all minima of $T$ down below all critical values of $\phi|_D$ into a neighborhood of $D^2 \times \{-1\}$. This isotopy fixes $\partial D$ and fixes the height and number of saddles of $D$. After this isotopy, each component of the intersection between $D$ and a neighborhood of $D^2 \times \{1\}$ is either a vertical boundary compressing disk or a compressing disk with level boundary and a single critical point corresponding to a maximum. Similarly, $D$ meets a neighborhood of $D^2 \times \{-1\}$ in a collection of vertical boundary compressing disks and compressing disks with level boundary and a single critical point corresponding to a minimum.

Let $a \in [-1,1]$ be the level of the highest local minimum of $T$ (with respect to $\phi$) and let $b$ be the level of the lowest local maximum. Then for every $t \in (a,b)$, the disk $\phi^{-1}(t)$ is isotopic to $\Sigma$ by an isotopy transverse to $T$. If for any $t \in (a,b)$ $\phi^{-1}(t)$ meets $D$ in a collection of arcs and loops that are all inessential in $\phi^{-1}(t)$, then $D$ can be isotoped to be disjoint from $\phi^{-1}(t)$. Hence, we will assume that for every $t \in (a,b)$  $\phi^{-1}(t)$ meets $D$ in at least one component that is essential in $\phi^{-1}(t)$.

For each fixed $t$, an innermost loop essential in $\phi^{-1}(t)$ or an outermost arc essential in $\phi^{-1}(t)$ of $D \cap \phi^{-1}(t)$ in $D$ will bound a subdisk of $D$ that can be isotoped to be either a compressing or boundary compressing disk for $\phi^{-1}(t)$. For $t$ near $a$, at least one of these disks will be below $\phi^{-1}(t)$, while for $t$ near $b$, at least one of these will be above $\phi^{-1}(t)$. If there is a value $t \in (a,b)$ with loops or arcs of intersection defining disks both above and below then these disks are disjoint (since $D$ is embedded) and we have one of conditions (1) or, in the case that both disks are boundary compressing disks, conclusion (2). Otherwise, there must be a critical point of $\phi|_D$ at which the disks switch from below to above. However, in this case the loops or arcs above and below the critical point correspond to disjoint arcs or loops in $\Sigma$, so we obtain conclusion (1) or, in the case that both disks are boundary compressing disks, conclusion (3).
\end{proof}

\begin{lemma}
\label{lem:weakly reducible2}
Let $(B,T)$ be a tangle with $B=(D^2\times [-1,1])$ such that $\partial T \subset (D^2\times \{-1,1\})$ and let $\Sigma$ be a bicompressible disk bridge surface for $T$. Let $\phi: D^2 \times [-1,1] \rightarrow [-1,1]$ be the natural projection map. Suppose $D$ is a compressing disk for $\bdd B$ in $B\setminus T$ such that $\partial D$ can be isotoped to be disjoint from $\partial \Sigma$. Assuming $\partial D \cap \partial \Sigma=\emptyset$, there is an isotopy of $D$ that fixes $\partial D$ and the height and number of saddles of $D$ after which one of the following occurs:
\begin{enumerate}

\item There exists a regular value $t$ of $\phi|_{D}$ such that $\phi^{-1}(t)$ is isotopic to $\Sigma$ and there are two disjoint compressing disks for $\phi^{-1}(t)$  on opposite sides of $\phi^{-1}(t)$,

\item There exists a critical value $t$ of $\phi|_{D}$ such that $\phi^{-1}(t)$ is isotopic to $\Sigma$ and there are two disjoint compressing disks for $\phi^{-1}(t)$ with boundaries contained in the boundary of a regular neighborhood of the component of $\phi^{-1}(t)\cap D$ that contains a saddle and on opposite sides of $\phi^{-1}(t)$,

\item For some $i\in \{-1,1\}$, $T\cap \phi^{-1}(i)=\emptyset$,

\item There is a strict subdisk of $\Sigma$ that is bicompressible.

\end{enumerate}
\end{lemma}

\begin{proof} Without loss of generality, we can isotope $\partial D$ to be contained in $\phi^{-1}(1)$ via an isotopy supported in a regular neighborhood of $\partial B$. Note that there can be no arcs of intersection between $D$ and $\phi^{-1}(t)$ for any $t$. Isotope all maxima of $T$ up above all critical points of $\phi|_D$ into a neighborhood of $D\times \{1\}$ and isotope all minimum of $T$ down below all critical points of $\phi|_D$ into a neighborhood of $D^2 \times \{-1\}$. This isotopy fixes $\partial D$ and fixes the height and number of saddles of $D$. After this isotopy, the collection of components of the intersection between $D$ and a neighborhood of $D^2 \times \{1\}$ is the union of a single vertical annulus with compressing disks, each having level boundary and a single critical point which is a maximum. Similarly, $D$ meets a neighborhood of $D^2 \times \{-1\}$ in a collection of compressing disks with level boundary and a single critical point corresponding to a minimum.

If the intersection between $D$ and a neighborhood of $D^2 \times \{1\}$ contains a compressing disk, then by the proof of Lemma \ref{lem:weakly reducible} we can conclude (1) or (2) holds.

If the intersection between $D$ and a neighborhood $D^2 \times [a,1]$ of $D^2 \times \{1\}$ is a single vertical annulus $A$ , then $D$ meets the portion of $B$ below $\phi^{-1}(a)$ in a single compressing disk $D^*$. If $\partial D^*$ is isotopic to $\partial \phi^{-1}(a)$, then (3) holds, so we can assume that $\partial D^*$ is not isotopic to $\partial \phi^{-1}(a)$. Since $\Sigma$ is bicompressible, let $E_1$ be a compressing disk for $\phi^{-1}(a)$ above $\phi^{-1}(a)$ and let $E_2$ be a compressing disk for $\phi^{-1}(a)$ below $\phi^{-1}(a)$. Boundary compress $E_1$ along $A$ and boundary compress $E_2$ along $D^*$ to produce two new compressing disks for $\phi^{-1}(a)$ each of which is disjoint from $\partial D^*$. If the boundaries of these disks lie in distinct components of the complement of $\partial D^*$ in $\phi^{-1}(a)$, then conclusion (1) holds. If the boundaries of these disks lie in a common component of the complement of $\partial D^*$ in $\phi^{-1}(a)$, then conclusion (4) holds.
\end{proof}

The following is known as the pop-over lemma and has been used in a number of papers. We refer the reader who is interested in the proof to Lemma 2.3 of~\cite{BT1}.

\begin{lemma}
\label{lem:popover}
Suppose that a tangle $T \subset D^2 \times [0,1]$ has $\boundary T \subset D^2 \times \{0\}$ and, under the projection to $[0,1]$ only has maxima. Then there is an isotopy $\psi$ fixing $\boundary (D^2 \times [0,1])$ so that $\psi(T)$ has the same number of maxima as $T$ and so that there is a strand $\alpha$ of $T$ that contains the highest maximum of $T$, meets a collar neighborhood $N$ of $\boundary (D^2 \times [0,1])$ in two vertical arcs and a horizontal arc and meets the complement of $N$ in $D^2 \times [0,1]$ in a vertical arc as in Figure \ref{fig:popover}.
\end{lemma}

\begin{figure}[htb]
  \begin{center}
  \includegraphics[width=1.5in]{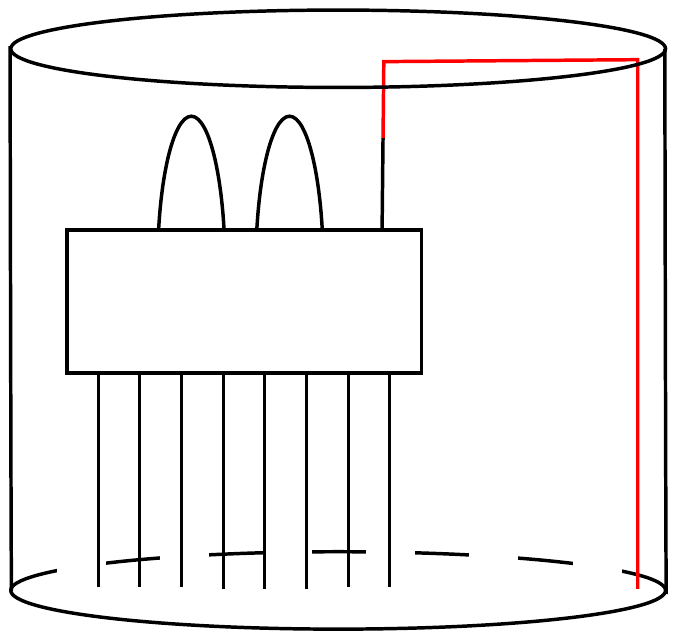}
  \caption{}
  \label{fig:popover}
  \end{center}
\end{figure}

\section{Distance 2 links}
\label{sect:proof}

\begin{proof}[Proof of Theorem~\ref{thm:main}]
Let $B^{\up}$ and $B^{\dn}$ be the two balls bounded by $\Sigma$ in $S^3$. Let $\gamma_1, \gamma_2, \gamma_3$ be a sequence of essential simple closed curves in $\Sigma$ such that $\gamma_i$ is disjoint from $\gamma_{i+1}$ for $i=1,2$ and such that $\gamma_1$ bounds a compressing disk in $B^{\up}$ and $\gamma_3$ bounds a compressing disk in $B^{\dn}$. The curve $\gamma_2$ separates $\Sigma$ into two disks. If $\gamma_1$ and $\gamma_3$ are in distinct disk components of $\Sigma \setminus \gamma_2$ then they are disjoint and $d(\Sigma) \leq 1$. Thus $\gamma_1$ and $\gamma_3$ must be contained in a single disk $F \subset \Sigma$ with boundary $\gamma_2$. Let $F^c$ be the closure of the complement of $F$ in $\Sigma$. In general, there may be many different distance-two paths for $\Sigma$ and we will assume that $\gamma_1$, $\gamma_2$ and $\gamma_3$ have been chosen amongst all such triples of curves so that $|F \cap L|$ is minimal.

Let $C$ be the boundary of a regular neighborhood $H^{in}$ of $F$. Isotope $L$ relative to $\Sigma$ so as to minimize $|C\cap L|$ without creating any new minima or maxima, as in Figure \ref{fig:general form}. Let $H^{out}$ be the closure of the complement in $S^3$ of $H^{in}$. Note that $H^{in}$ and $H^{out}$ are both balls bounded by the sphere $C$ such that $H^{in}$ contains $F$ and $L \cap H^{in}$ is a tangle in $H^{in}$. Moreover, we can naturally identify $H^{in}$ with $D^2 \times [-1,1]$ so that $F$ is a disk bridge surface for $(H^{in}, L\cap H^{in})$.
\begin{figure}[htb]
  \begin{center}
  \includegraphics[scale=0.6]{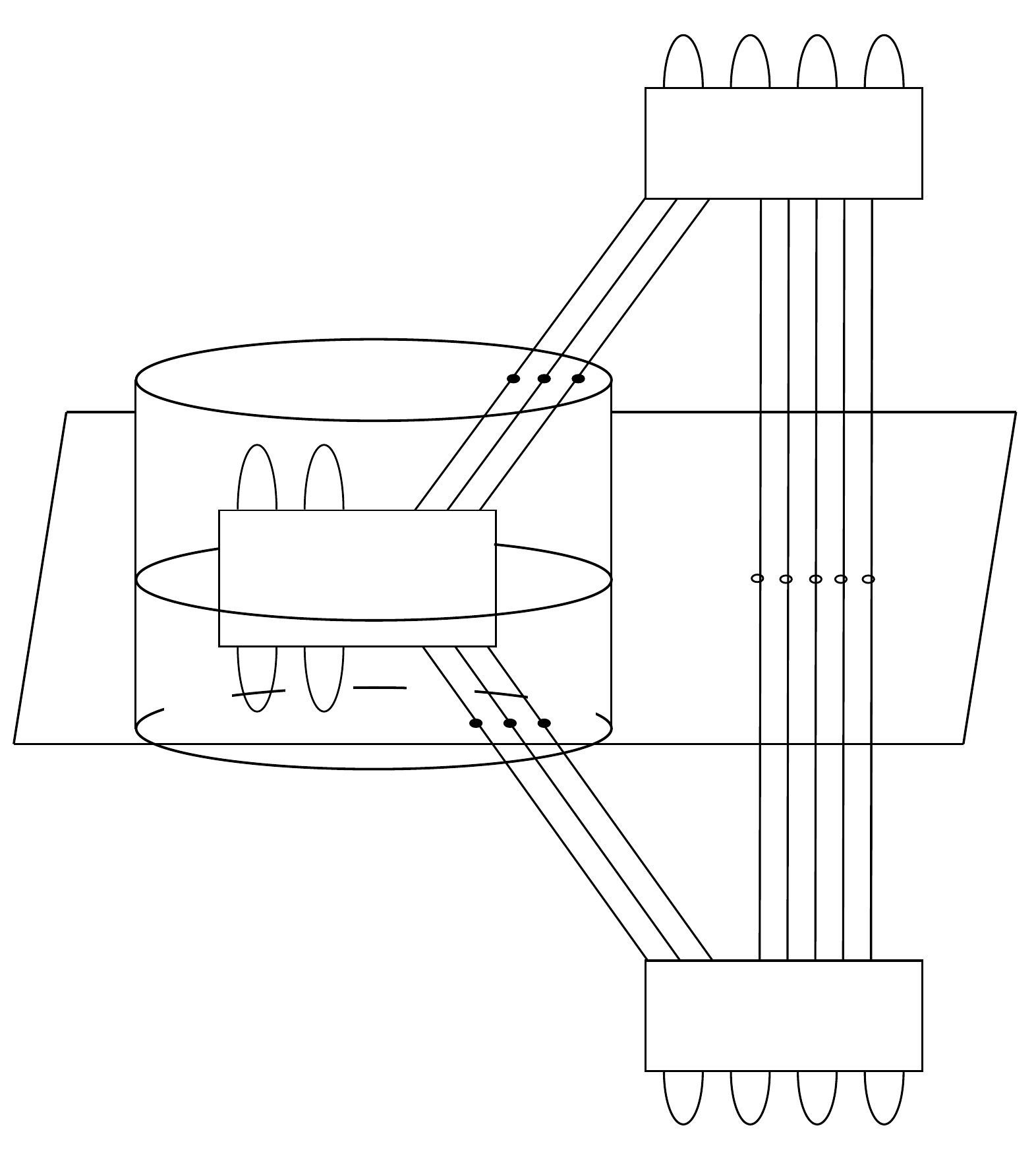}
  \put(-35,150){$\Sigma$}
  \put(-30,230){$B^{\up}$}
  \put(-30,70){$B^{\dn}$}
  \put(-180,200){$C$}
  \put(-130,150){$F$}
  \put(-100,150){$F^c$}
  \put(-230,180){$H^{in}$}
  \put(-280,224){$H^{out}$}
  \caption{}
  \label{fig:general form}
  \end{center}
\end{figure}

\medskip

{\bf Claim 0:} $F^c$ is incompressible in $H^{out}$.

\emph{Proof of Claim 0:} Suppose that $E$ is a compressing disk for $F^c$. Since both $\gamma_1$ and $\gamma_2$ are disjoint from $F^c$ then there is a compressing disk for $\Sigma$ that is disjoint from $E$ and on the opposite side of $\Sigma$ from $E$. Hence, $\Sigma$ is distance at most one, a contradiction.
\qed (Claim 0)

\medskip

{\bf Claim 1:} $C \cap B^{\up}$ is incompressible in $B^{\up}$ and $C \cap B^{\dn}$ is incompressible in $B^{\dn}$.

\emph{Proof of Claim 1:} Assume for contradiction $D$ is a compressing disk for $C$ in $B^{\up}$. Then $\partial D$ also bounds a disk $\Delta$ in $C\cap B^{\up}$. The union $D \cup \Delta$ is a ball $B'$ and each arc of $L$ in $B'$ contains exactly one maximum since $\Sigma$ is a bridge sphere for $L$. Hence, $L \cap B'$ is a trivial tangle in $B'$. Moreover, the endpoints of each strand of $L \cap B'$ are in $\Delta$, so we can isotope a strand $\beta$ of $L\cap B'$ into $\Delta$. After pushing $\beta$ just past $\Delta$ and into or out of  $H^{in}$, $C$ will meet $L$ in two fewer points, contradicting our minimality assumption. The proof that $C \cap B^{\dn}$ is incompressible in $B^{\dn}$ is essentially the same.
\qed (Claim 1)

\medskip

{\bf Claim 2:} As a properly embedded surface in $H^{in}$, $F$ does not have a compressing disk on one side that is disjoint from a compressing or a boundary compressing disk on the opposite side.

\emph{Proof of Claim 2:} If there are two disjoint compressing disks on opposite sides of $F$ then $\Sigma$ is a bridge sphere of $L$ with distance at most one, contrary to our hypothesis.

Suppose there is a compressing disk $\bar{D}$ for $F$ on one side, say in $B^{\up}\cap H^{in}$, that is disjoint from a boundary compressing disk $\Delta$ for $F$ in $B^{\dn}\cap H^{in}$, as in Figure \ref{fig:reduction}. Let $D_3$ be the compressing disk that $\gamma_3$ bounds in $B^{\dn}$. By Claim 1, we may assume that this disk is disjoint from $C$ and therefore contained in $B^{\dn}\cap H^{in}$. We can eliminate intersections between $\Delta$ and $D_3$ by compressing and boundary compressing $D_3$ along $\Delta$. One of the resulting components will be a compressing disk $D^*$ for $F$ in $B^{\dn}\cap H^{in}$ that is disjoint from $\Delta$. If $\bar{D}$ is on the opposite side of $\Delta$ from $D^*$, then $D^*$ and $\bar D$ are disjoint so the distance of $\Sigma$ is at most one, contradicting the assumption. Thus, $\bar{D}$ and $D^*$ are on the same side of $\Delta$.
\begin{figure}[htb]
  \begin{center}
  \includegraphics[scale=0.4]{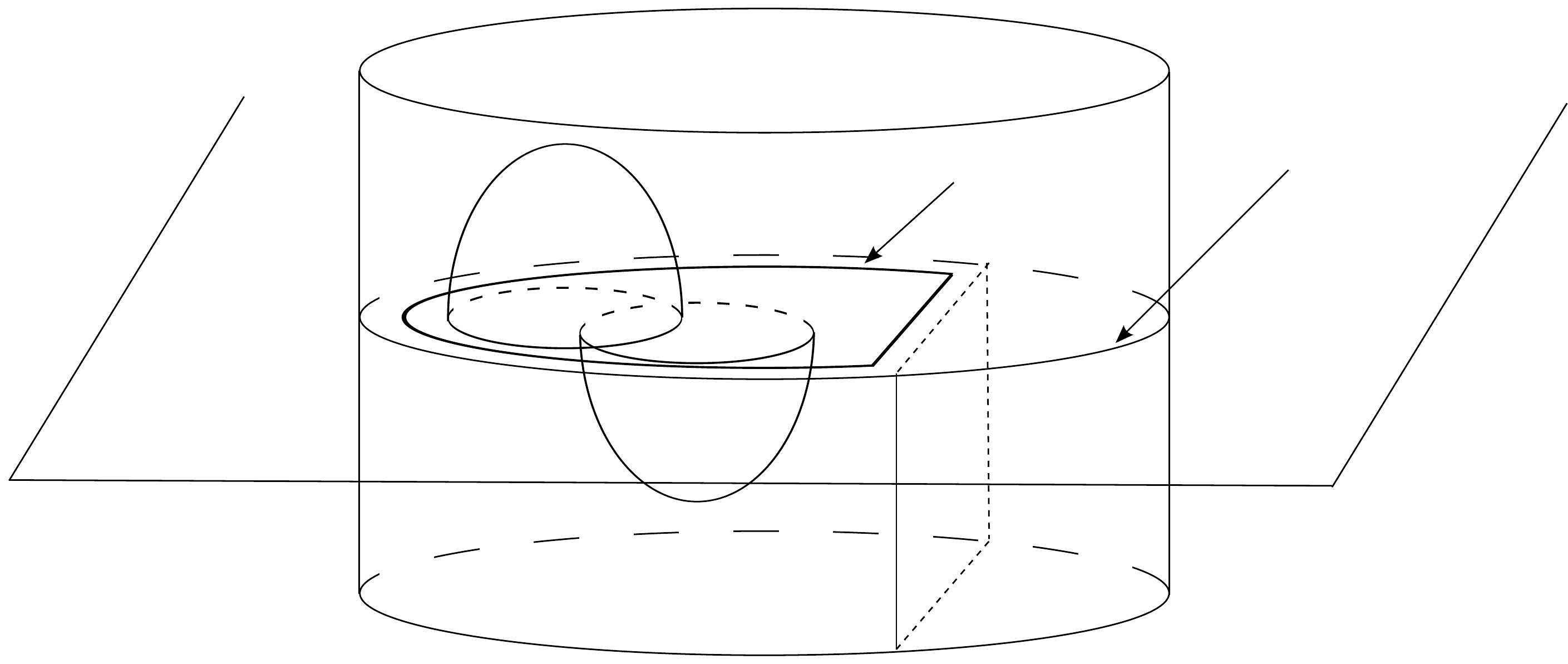}
  \put(-212,90){$\bar{D}$}
   \put(-133,40){$\Delta$}
   \put(-185,43){$D^*$}
      \put(-132,105){${\gamma_2}'$}
       \put(-60,110){${\gamma_2}$}
  \caption{}
  \label{fig:reduction}
  \end{center}
\end{figure}

Consider the curve ${\gamma_2}'$ that is contained in the boundary of a regular neighborhood of $\gamma_2 \cup (\Delta \cap F)$ in $F$ and separates $\gamma_2$ from $\partial D^*$, as in Figure \ref{fig:reduction}. Let $F'$ be the disk in $\Sigma$ bounded by $\gamma'_2$ that contains $\bdd \bar{D}$ and $\partial D^*$. Then $|F' \cap L| < |F \cap L|$, contradicting the minimality assumption on the path $\gamma_1, \gamma_2, \gamma_3$.
\qed(Claim 2)

\medskip

\begin{remark}\label{minsandmaxes}
Note that if $F$ has disjoint boundary compressing disks $\Delta^{\up}$ and $\Delta^{\dn}$ on opposite sides of $F$, then we know where the critical points of the tangle $T \cap H^{in}$ occur with respect to $\Delta^{\up}$ and $\Delta^{\dn}$. If the tangle has maxima on both sides of $\Delta^{\up}$ then we can find compressing disks associated to each of these maxima that are disjoint from $\Delta^{\up}$ and occur on both sides of $\Delta^{\up}$ by taking the frontier of regular neighborhoods of bridge disks that have been chosen to be disjoint from $\Delta^{\up}$. Since $\Delta^{\dn} \cap F$ is on one side or the other of $\Delta^{\up}\cap F$ in $F$, this implies that at least one of these compressing disks is disjoint from $\Delta^{\dn}$. Thus by Claim 2, every maximum of the tangle must be on the same side of $\Delta^{\up}$ as $\Delta^{\dn}$. Symmetrically, every minimum of the tangle must be on the same side of $\Delta^{\dn}$ as $\Delta^{\up}$.
\end{remark}

\medskip

{\bf Claim 3:} If $C$ is compressible in $H^{out}$, then $L$ is obtained from a link $L'$ via banding where $L'$ has a bridge sphere with the same bridge number as $\Sigma$, but with distance at most one.

\emph {Proof of Claim 3:}
Suppose $D$ is a compressing disk for $C$ contained in $H^{out}$ that intersects $\Sigma$ minimally amongst all such disks. If $D \cap \Sigma$ contains a simple closed curve, then an innermost loop of intersection in $D$ bounds a compressing disk for $\Sigma$ that is disjoint from $\gamma_1$ and $\gamma_3$, so the distance of $\Sigma$ would be at most one, contradicting our hypothesis. Also, by Claim 1, $D \cap \Sigma \neq \emptyset$. Thus, we can conclude that $D \cap \Sigma$ is a non-empty collection consisting entirely of arcs.
\begin{figure}[htb]
  \begin{center}
  \includegraphics[scale=0.18]{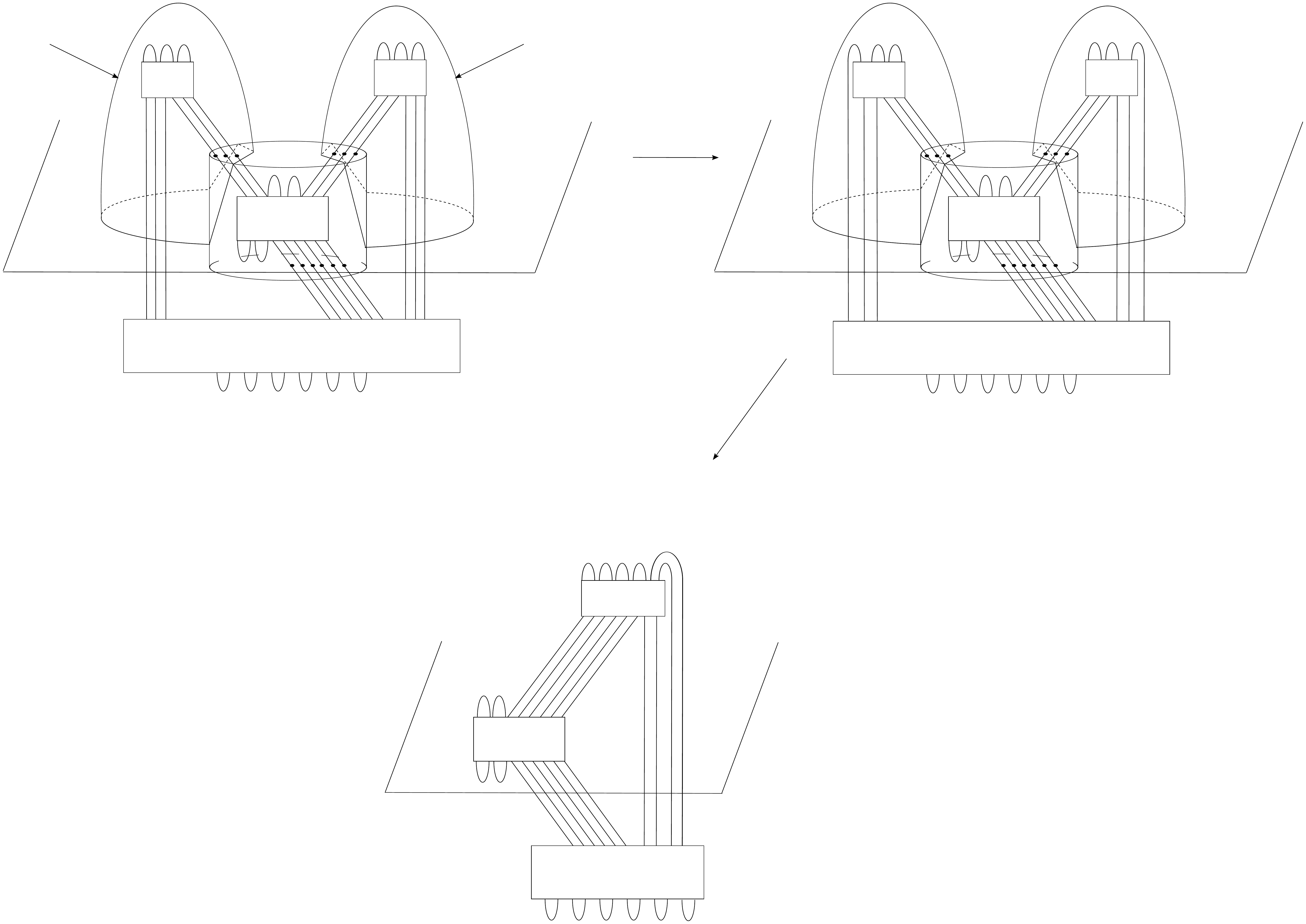}
    \put(-343,225){$\Delta_1$}
   \put(-200,225){$\Delta_2$}
     \put(-310,200){$B_1$}
   \put(-227,200){$B_2$}
  \caption{}
  \label{fig:BoundaryCompression}
  \end{center}
\end{figure}

Let $\Delta$ be any compressing disk for $\partial (B^{\up} \cap H^{out})$ or for $\partial (B^{\dn} \cap H^{out})$ that decomposes $B^{\up} \cap H^{out}$ or $B^{\dn} \cap H^{out}$ into two 3-balls $B_1$ and $B_2$ such that both $\partial B_1$ and $\partial B_2$ have non-trivial intersection with $L\cap \Sigma$. In the current situation, we could let $\Delta$ be a disk cut from $D$ by an outermost arc $\alpha$ of $D \cap \Sigma$. However, we will need the full generality of the definition of $\Delta$ in Claim 5.

Choose a collection of bridge disks for the arcs of $L$ above $\Sigma$ which is disjoint from $\Delta$. Use these bridge disks to lower all bridges of $L$ above $\Sigma $ below any critical points of $\Delta$. Now we can isotope $\Delta$ to have a single maximum by an isotopy that does not affect the link. We can then isotope $L$ to its original position without introducing any new critical points in $\Delta$. Therefore we may assume that after an isotopy $\Delta$ has a single maximum in its interior.

Let $\Delta_1$ and $\Delta_2$ be two copies of $\Delta$ with the point at infinity between them, as in the top row of Figure \ref{fig:BoundaryCompression}.  After thickening $\Delta$ and an isotopy, $B_1$ and $B_2$ are the balls indicated in the top left of the Figure bounded by $\Delta_1$ and $\Delta_2$, respectively, and disjoint disks in $\Sigma$. Then $T_i=B_i \cap L$ is a tangle that only has maxima so we can apply Lemma \ref{lem:popover} to pull one strand with at least one endpoint in $\Sigma$ out of the braid in each of $B_1$ and $B_2$ as indicated in the upper right picture of Figure \ref{fig:BoundaryCompression}. Our definition of $\Delta$ guarantees that such strands exist. These two strands then form a band as indicated in the last picture. The link $L'$ produced by undoing this band, the inverse of the process depicted in Figure \ref{fig:banding}, continues to have bridge surface $\Sigma$. However, $\Sigma$ now has an obvious pair of disjoint compressing disks on opposite sides. Therefore $L'$ has a bridge surface with the same bridge number as $L$, but distance at most one. (Note that this bridge surface may be the result of perturbing a bridge surface with strictly lower bridge number.)
\qed(Claim 3)

\medskip

Both Claim 4 and Claim 5 derive consequences from the assumption that $C$ is compressible in $H^{in}$.  It is necessary to use the conclusions of these claims in tandem to prove the theorem.

\medskip

{\bf Claim 4:} If $C$ is compressible in $H^{in}$, then $H^{in} \cap L$ is a bridge-split Montesinos tangle or $L$ is obtained from a link $L'$ via banding where $L'$ has a bridge sphere with the same bridge number as $\Sigma$, but with distance at most one.

\emph{Proof of Claim 4:} Suppose $D$ is a compressing disk for $C$ in $H^{in}$. If $\partial D$ can be isotoped to be disjoint from $\partial F$, then we can apply Lemma \ref{lem:weakly reducible2} and we arrive at a contradiction to Claim 2 if conclusion (1) or (2) holds. If (3) of Lemma \ref{lem:weakly reducible2} holds, then we contradict Claim 0. If (4) of Lemma \ref{lem:weakly reducible2} holds, then we contradict how we chose the curves $\gamma_1, \gamma_2, \gamma_3$ by finding a smaller bicompressible subdisk of $\Sigma$.

Thus, we can assume that $\partial D$ cannot be isotoped to be disjoint from $\partial F$. Apply Lemma \ref{lem:weakly reducible} to conclude that there are two cases to consider. If conclusion (1) of Lemma \ref{lem:weakly reducible} holds, then we contradict Claim 2. Thus, we may assume conclusion (2) or (3) holds. Hence, $F$ has a pair of disjoint boundary compressing disks on opposite sides. Let $\mc{D}^{\up}$ and $\mc{D}^{\dn}$ be these disks such that $\mc{D}^{\up}$ lies above $F$ and $\mc{D}^{\dn}$ lies below $F$.   By Lemma \ref{lem:vertical boundary compressing disks}, we may assume that these disks are vertical and by Remark \ref{minsandmaxes}, every compressing disk disjoint from $\mc{D}^{\up}$ must be contained in the component of $F \setminus \mc{D}^{\up}$ that contains $\mc{D}^{\dn}$ and vice versa. Let $R'$ be the result of boundary compressing $F$ along $\mc{D}^{\up}\cup \mc{D}^{\dn}$. By the above argument, some component $R_0$ of $R'$ must separate all the minima from all the maxima of the tangle.

Assume $R_0$ is boundary parallel in $H^{in}$. This boundary parallelism induces a product structure of the form $R_0\times I$ on at least one of the two components of $H^{in}\setminus R_0$. Since $F$ is bicompressible, some strand of $L \cap H^{in}$ contains a maximum and some strand contains a minimum. Without loss of generality, the component of $H^{in}\setminus R_0$ that contains maxima, $F^{\up}$ inherits the product structure. Let $\alpha$ be a strand of $L \cap F^{\up}$ that contains a maximum. The product structure of $F^{\up}$ allows for a level preserving isotopy of $L$ that fixes $L$ away from $\alpha$ and results in $\alpha$ being pulled out of $H^{in}$ and into $H^{out}$, as in leftmost isotopy of Figure \ref{fig:R_0BoundaryPara}. Apply Lemma \ref{lem:popover} to pull one strand with at least one endpoint in $F^c$ out of $L\cap B^{\up}\cap H^{out}$, as in the middle isotopy in Figure \ref{fig:R_0BoundaryPara}. This strand and $\alpha$ cobound a band, as illustrated in the right most isotopy in Figure \ref{fig:R_0BoundaryPara}. Hence, $L$ is obtained from a link $L'$ via banding where $L'$ has a bridge sphere with the same bridge number as $\Sigma$, but with distance at most one. Since we arrive at the same conclusion independent of which component of $H^{in}\setminus R_0$ illustrates the boundary parallelism, we can assume that $R_0$ is not boundary parallel in $H^{in}$.

\begin{figure}[htb]
  \begin{center}
  \includegraphics[width=5in]{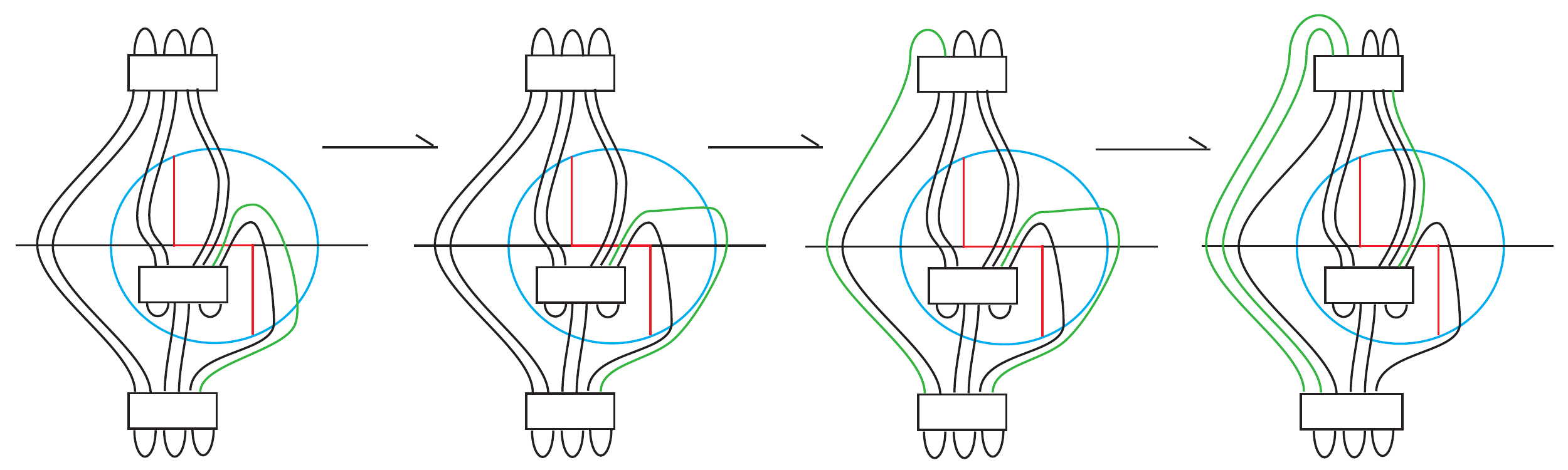}
  \caption{If $R_0$ is boundary parallel, then $L$ is banded.}
  \label{fig:R_0BoundaryPara}
  \end{center}
\end{figure}

We can think of $R_0$ as a subdisk of $F$ together with vertical flaps corresponding to $\mc{D}^{\up}$ and $\mc{D}^{\dn}$, as in Figure~\ref{fig:extending}. Hence, after a small tilt of the portion of $R_0$ in $F$, the foliation of $R_0$ induced by the restriction of the projection of $H^{in}\cong D^2 \times I$ onto its $I$-factor is a collection of parallel arcs. To demonstrate that $H^{in} \cap L$ is a bridge-split Montesinos tangle, it only remains to show that $R_0$ is incompressible in $H^{in}$.

Suppose $E$ is a compressing disk for $R_0$. After an isotopy, $E$ becomes a compressing disk for $F$ with boundary disjoint from $\alpha=\mc{D}^{\up}\cap F$ and $\beta=\mc{D}^{\dn}\cap F$. Without loss of generality, assume $E$ is above $F$. Since $F$ is bicompressible in $H^{in}$ by construction, there is a compressing disk for $F$ contained below $F$. By compressing and boundary compressing this compressing disk along subdisks of $\mc{D}^{\dn}$, we obtain a compressing disk $E^b$ for for $F$ contained below $F$, with boundary disjoint from $\beta$. $\partial E^b$ must be contained in the same component of $F\setminus \beta$ as $\alpha$, otherwise we contradict Claim 2. Since the boundary of $E^b$ and $\partial E$ both are contained in a subdisk of $F$ that meets $L$ in strictly fewer points, we arrive at a contradiction to how we chose $\gamma_1, \gamma_2, \gamma_3$. Thus, $R_0$ is incompressible.
\begin{figure}[htb]
  \begin{center}
  \includegraphics[scale=0.35]{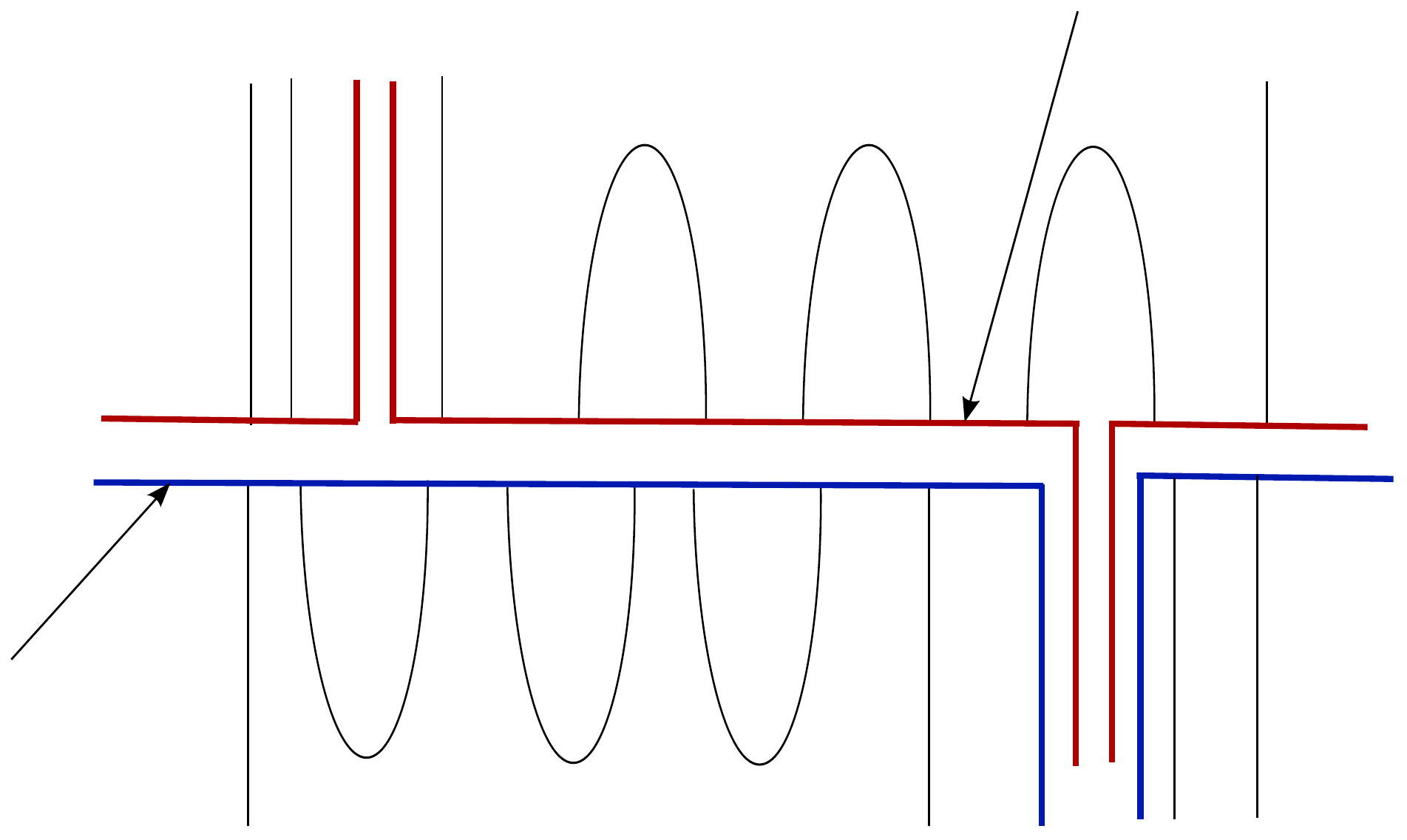}
   \put(-50,120){$R_0$}
  \caption{}
  \label{fig:extending}
  \end{center}
\end{figure}

\qed(Claim 4)

\medskip

{\bf Claim 5:} If $C$ is compressible in $H^{in}$, then at least one of the following holds:

\begin{enumerate}
\item the exterior of $L$ contains an incompressible meridional planar surface,

\item $H^{in} \cap L$ is rational,

\item $L$ is obtained from a link $L'$ via banding where $L'$ has a bridge sphere with the same bridge number as $\Sigma$, but with distance at most one,

\item the exterior of $L$ in $H^{out}$ contains an incompressible, boundary incompressible annulus.
\end{enumerate}

The proof of the following claim is inspired by the work of Rubinstein and Thompson~\cite{RT} on the classification of distance-two Heegaard splittings and the proof of Lemma 3.2 in~\cite{CR}.

\emph {Proof of Claim 5:} By Claim 3, we can assume that $C$ is incompressible in $H^{out}$. Let $\mc{E}$ be a maximal collection of compressing disks for $C$ in $H^{in}$. Let $N$ be the regular neighborhood of $C\cup \mc{E}$. If any of the boundary components of $N$ are 2-spheres bounding balls in $H^{in}$ disjoint from $L$ or are boundary parallel twice punctured 2-spheres, then fill these components with the corresponding 3-balls or 3-balls containing unknotted arcs to form the 3-manifold $M$. Let $\partial_+ M=\partial M -C$. If $\partial_+ M= \emptyset$, then $H^{in} \cap L$ is rational and the claim holds. Hence, we can assume that $\partial_+ M \neq \emptyset$. Since $\mc{E}$ was maximal, $\partial_+ M$ is incompressible to the side not containing $C$. $M$ is a punctured 3-ball that meets $L$ in a collection of arcs isotopic into $C$ or connecting $C$ to $\partial_+ M$. Thus, $\partial_+ M$ is incompressible in $M-L$.

Suppose that $\partial_+ M$ is compressible in the exterior of $L$ with compressing disk $D$. Isotope $D$ to be transverse to $C\cup F^c$ so that $D \cap (C\cup F^c)$ consists of loops and 3-valent graphs in $D$. Assume that we have isotoped $D$ to meet $C\cup F^c$ in a minimal number of edges, where we think of each loop component as consisting of a single vertex and a single edge. Since $\partial_+ M$ is incompressible in $M-L$, $D \cap (C\cup F^c)\neq \emptyset$. A component $\Gamma$ of $D \cap (C\cup F^c)$ in $D$ is innermost in $D$ if all but one of the boundary components of a closed regular neighborhood of $\Gamma$ in $D$ bound disks in $D$ that are disjoint from $C\cup F^c$. Let $\Gamma$ be an innermost component of $D \cap (C\cup F^c)$ in $D$.

\textbf{Case 1:} Suppose that $\Gamma$ is a loop. Since $\Gamma$ is innermost it bounds a disk $D^*$ in $D$ which is disjoint from $(C\cup F^c)$. If $D^*$ is contained in $H^{out}$, then $\Gamma$ is disjoint from $\partial F^c$ since it is a loop and not a 3-valent graph. If $\Gamma$ is contained in $F^c$, then we obtain a contradiction to Claim 0. Thus, we can assume that $\Gamma \subset C$. If $D^*$ is contained in $H^{out}$, then $D^*$ is a compressing disk for $C \cap B^{\up}$ in $B^{\up}$ or $C \cap B^{\dn}$ in $B^{\dn}$, a contradiction to Claim 1. Hence, $D^*$ is contained in $H^{in}$. Given the natural identification of $H^{in}$ with $D^2\times I$ and the fact that $\Gamma$ is disjoint from $\partial F$, then we can assume without loss of generality that $D^*$ has boundary in $D^2\times \{1\}$. Apply Lemma \ref{lem:weakly reducible2} to $D^*$ and analyze each possible conclusion. If conclusions (1) or (2) of Lemma \ref{lem:weakly reducible2} hold, then we contradict the fact that $\Sigma$ is distance two. If conclusion (3) of Lemma \ref{lem:weakly reducible2} holds, then we contradict Claim 0. If conclusion (4) of \ref{lem:weakly reducible2} holds, we contradict how we chose $\gamma_1, \gamma_2, \gamma_3$. Thus, $\Gamma$ can not be a loop.

\textbf{Case 2:} Suppose that $\Gamma$ is a 3-valent graph. Since $\Gamma$ is embedded in the interior of the disk $D$, the formula for Euler characteristic gives $1=v(\Gamma)-e(\Gamma)+(|D-\Gamma|-1)$ where $v(\Gamma)$ is the number of vertices of $\Gamma$, $e(\Gamma)$ is the number of edges of $\Gamma$ and $|D-\Gamma|$ counts the number of components of the complement of $\Gamma$ in $D$. Since $\Gamma$ is 3-valent, $\frac{2}{3}e(\Gamma)=v(\Gamma)$. Thus, $2=-\frac{1}{3}e(\Gamma) + |D-\Gamma|$. From this equality follows the inequality

$$\frac{e(\Gamma)}{|D-\Gamma|}<3.$$

Since every edge of $\Gamma$ borders exactly two faces and since the closure of each component of $|D-\Gamma|$ has boundary a cycle in $\Gamma$ consisting of an even number of edges, then the above inequality implies the existence of a disk component of $|D-\Gamma|$ with closure $D^*$ such that $\partial D^*$ consists of exactly two edges and two vertices of $\Gamma$ ($D^*$ is a bigon), or $\partial D^*$ consists of exactly four edges and four vertices of $\Gamma$ ($D^*$ is a square).

Suppose $D^*$ is a bigon properly embedded in $H^{in}$. Since $D^*$ is a bigon, $D^*$ meets $F$ in exactly one arc $\alpha$ that is essential in $F$. Identify $H^{in}$ with $D^2 \times I$ and isotope $\partial D^*$ to consist of two vertical arcs in $\partial D^2 \times I$ and two horizontal arcs in $D^2\times \partial I$. Hence $D^*$ meets every level surface $D^2\times \{t\}$ in exactly one arc. Apply Lemma \ref{lem:weakly reducible} and the proof of Claim 4 to the compressing disk $D^*$ to conclude that $F$ has a pair of disjoint boundary compressing disks on opposite sides, $\mc{D}^{\up}$ and $\mc{D}^{\dn}$ such that the arc $F\cap \mc{D}^{\up}$ is isotopic to the arc $F\cap \mc{D}^{\up}$.

Since $F$ is bicompressible in $H^{in}$, we can boundary compress along $\mc{D}^{\up}$ to find a compressing disk for $F$ above $F$ which is disjoint from $\mc{D}^{\up}$. Similarly, we can find a compressing disk for $F$ below $F$ which is disjoint from $\mc{D}^{\dn}$. If these compressing disks have boundary on the same side of $F\cap \mc{D}^{\up}$ in $F$, then we contradict how we chose the curves $\gamma_1, \gamma_2, \gamma_3$. If these compressing disks have boundary on opposite sides of $F\cap \mc{D}^{\up}$ in $F$, then $\Sigma$ has distance at most one, a contradiction. Hence, $D^*$ is not a bigon properly embedded in $H^{in}$.

Suppose $D^*$ is a bigon properly embedded in $B^{\up}\cap H^{out}$ or $B^{\dn}\cap H^{out}$. Then $D^*$ satisfies the definition of $\Delta$ in the proof of Claim 3 and we conclude that $L$ is obtained from a link $L'$ via banding where $L'$ has a bridge sphere with the same bridge number as $\Sigma$, but with distance at most one.

Suppose $D^*$ is a square properly embedded in $H^{in}$. Identify $H^{in}$ with $D^2 \times I$ and, since $D^*$ is a square, we can isotope $\partial D^*$ to consist of four vertical arcs in $\partial (D^2) \times I$, two horizontal arcs in $D^2\times \{1\}$ and two horizontal arcs in $D^2\times \{0\}$. Recall that $\phi$ is the natural height function on $H^{in}$. At any regular value $t$ of $\phi|_{D^*}$, $\phi^{-1}(t)\cap D^*$ consists of exactly two arcs and some number of loops. Since $D^*$ is connected, there must be some saddle singularity of $\phi|_{D^*}$ at which two arcs are banded together.

Since $D^*$ is simply connected, there is at most one saddle where two arcs are banded together to form two new arcs. Let $a$ be the height of this singularity, then the graph $\phi^{-1}(a)\cap D^*$ consists of a single ``X'' and possibly a collection of loops. If any of the loops of intersection of $\phi^{-1}(a)\cap D^*$ are inessential in $\phi^{-1}(a)$, then we can eliminate them via an isotopy of $D^*$ that leaves the saddle singularity fixed. As in the proof of Claim 4, we can apply Lemma \ref{lem:weakly reducible} to find an isotopy of $D^*$ after which conclusion (2) or (3) of Lemma \ref{lem:weakly reducible} holds. Let $b$ be the value for which $D^*\cap \phi^{-1}(b)$ has a pair of disjoint boundary compressing disks on opposite sides, $\mc{D}^{\up}$ and $\mc{D}^{\dn}$.

Suppose conclusion (2) of Lemma \ref{lem:weakly reducible} holds. If $b>a$, then the arc $\beta = \mc{D}^{\dn}\cap \phi^{-1}(b)$ separates $D^*$ into two disks. Let $D'$ be the possibly immersed disk that is the union of $\mc{D}^{\dn}$ and the portion of $D^*$ that is cut off by $\beta$ and lies above $\phi^{-1}(b)$ when near $\beta$, see Figure~\ref{fig:Dprime}. Since $b>a$, $\partial D'$ consists of two vertical arcs in $\partial (D^2) \times I$, one horizontal arc in $D^2\times \{1\}$ and one horizontal arc in $D^2\times \{0\}$. Since the self-intersections of $D'$ correspond to the transverse intersections of embedded disks, then after surgering $D'$ along any loops of self intersection, we can assume $D'$ is an embedded compressing disk. In this case $D'$ is a bigon in $H^{in}$ and we can apply the above argument to derive a contradiction. Similarly, we derive a contradiction if $b<a$. Hence $b=a$. However, $b=a$ implies conclusion (3) rather that conclusion (2) of Lemma \ref{lem:weakly reducible}.

\begin{figure}[htb]
  \begin{center}
  \includegraphics[scale=0.3]{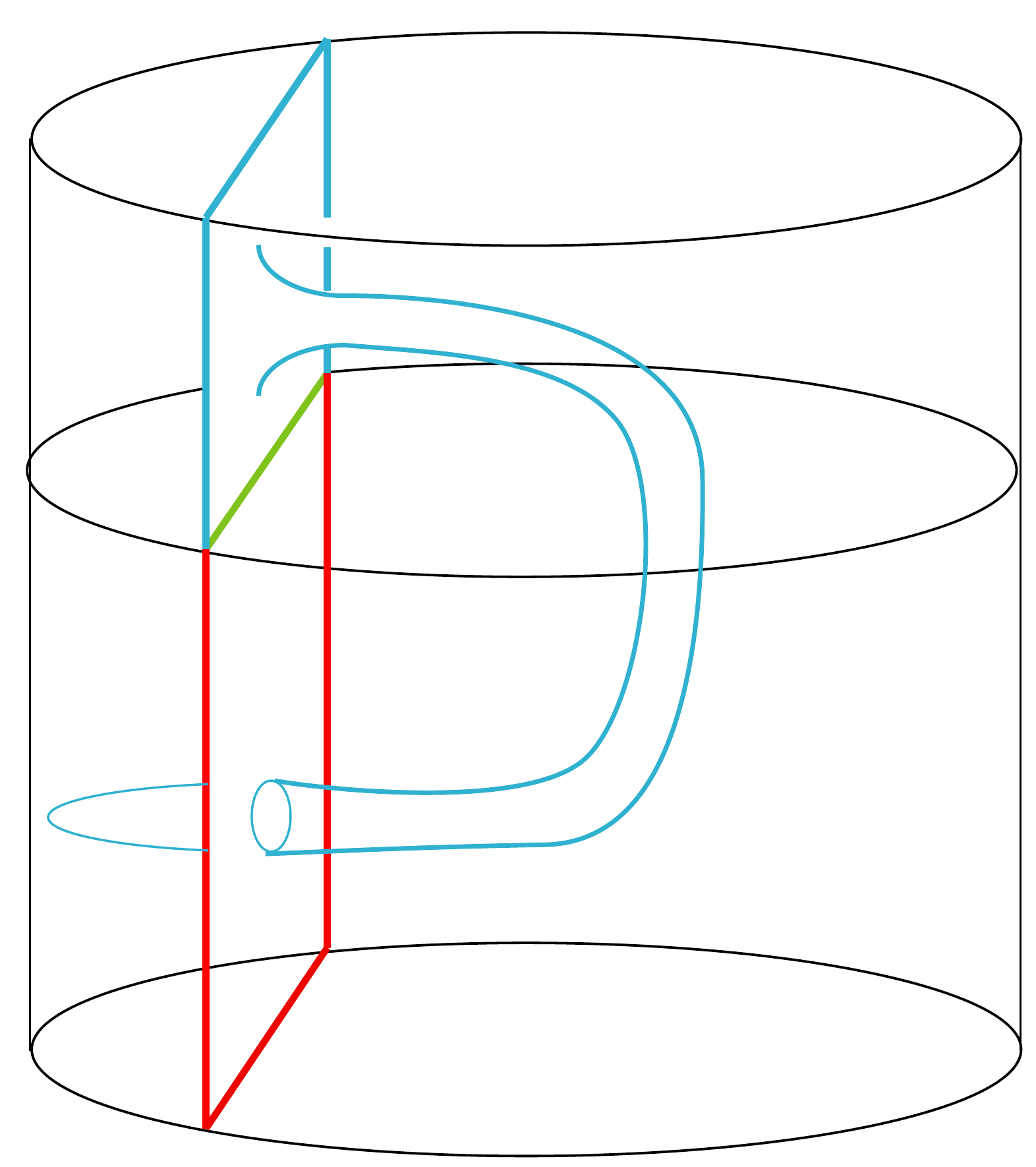}
  \put(-108,110){$D^{*}$}
  \put(-104,84){$\beta$}
  \put(-108,60){$\mc{D}^{\dn}$}
  \caption{The immersed disk $D'$}
  \label{fig:Dprime}
  \end{center}
\end{figure}

Suppose conclusion (3) of Lemma \ref{lem:weakly reducible} holds. If $b>a$ let $\beta = \mc{D}^{\dn}\cap \phi^{-1}(b-\epsilon)\subset D^*\cap \phi^{-1}(b-\epsilon)$ for $\epsilon$ small. We arrive at a contradiction when we consider the immersed disk $D'$ that is the union of $\mc{D}^{\dn}$ below $\phi^{-1}(b-\epsilon)$ and the portion of $D^*$ that is cut off by $\beta$ and lies above $\phi^{-1}(b-\epsilon)$ when near $\beta$ exactly as before. Similarly, we derive a contradiction if $b<a$. If $b=a$, then there are no loops in $D^*\cap \phi^{-1}(a)$ as we have already isotoped $D^*$ to eliminate loops that are inessential in $\phi^{-1}(a)$ and if any such loops are essential in $\phi^{-1}(a)$ we derive a contradiction as in Claim 4. Hence, $\phi^{-1}(a)\cap D^*$ consists of a single ``X.''

Identify $F$ with $\phi^{-1}(a)$ and let $x_1, x_2, x_3, x_4$ be the four properly embedded arcs in $F$ that lie in the boundary of a regular neighborhood of $F\cap D^*$ in $F$. By the above argument, each of these arcs bound a boundary compressing disk for $F$. In particular, $x_1$ and $x_3$ bound boundary compressing disks above $F$ and $x_2$ and $x_4$ bound boundary compressing disks below $F$. By Remark \ref{minsandmaxes}, all maxima of $L \cap H^{in}$ must be contained in the 3-ball between the upper boundary compressing disks in $H^{in}\cap B^{\up}$ and all minima of $L \cap H^{in}$ must be contained in the 3-ball between the lower boundary compressing disks in $H^{in}\cap B^{\dn}$. Hence $D^*$ decomposes $L\cap H^{in}$ into two rational tangles glued together along an unpunctured disk. See Figure \ref{fig:monkeyclasp2}. Thus, $L\cap H^{in}$ is a rational tangle.

\begin{figure}[htb]
  \begin{center}
  \includegraphics[scale=0.6]{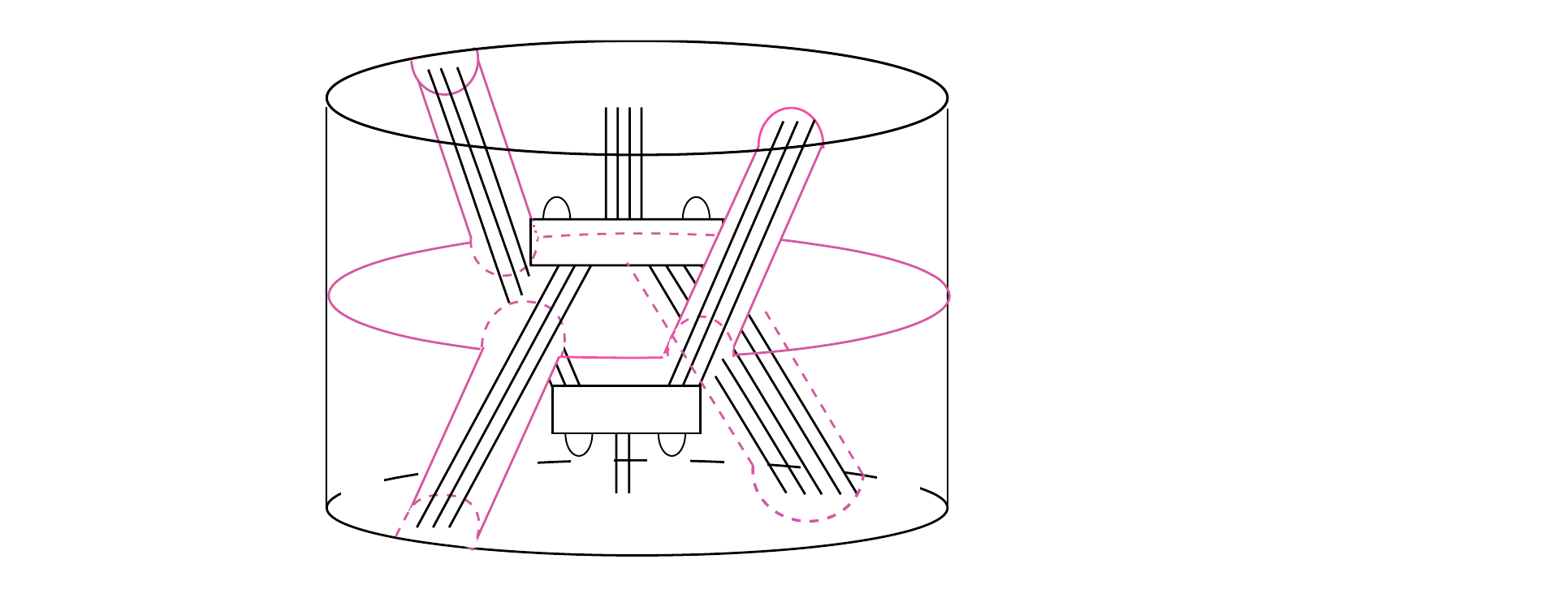}
  \caption{}
  \label{fig:monkeyclasp2}
  \end{center}
\end{figure}

Suppose $D^*$ is a square properly embedded in $B^{\up}\cap H^{out}$ or $B^{\dn}\cap H^{out}$. We consider the case that $D^*$ is a square properly embedded in $B^{\up}\cap H^{out}$ and the case when $D^*$ is a square properly embedded in $B^{\dn}\cap H^{out}$ will follow by a symmetric argument. Recall that $F^c$ is the closure of the complement of $F$ in $\Sigma$. Since $D^*$ is a square, it meets $F^c$ in exactly two essential arcs. If these arcs are not isotopic in $F^c$ then $D^*$ meets the definition of $\Delta$ in the proof of Claim 3 and we conclude that $L$ is obtained from a link $L'$ via banding where $L'$ has a bridge sphere with the same bridge number as $\Sigma$, but with distance at most one. Hence, we can assume that the two arcs in $F^c\cap D^*$ together with two arcs in $\partial F^c$ bound a disk $D^F$ in $F^c$ which is disjoint from $L$. After a small isotopy pushing $D^F$ into the interior of $B^{\up}$, $A=D^F\cup D^*$ is an annulus properly embedded in $B^{\up}\cap H^{out}$ with boundary in $C \cap B^{\up}$. If $A$ is compressible in the exterior of $L$ in $H^{out}$, then, after compressing $A$, each component of $\partial A$ bounds a compressing disk for $C \cap B^{\up}$ in $H^{out}$, a contradiction to the incompressibility of $C$ in $H^{out}$. Hence, $A$ is incompressible. $A$ is not boundary parallel since $\partial A$ partitions $C$ into three subsurfaces, each of which has non-trivial intersection with $L$. Lastly, $A$ is boundary incompressible since boundary compressing $A$ would result in a compressing disk for $C$ in $H^{out}$ which was ruled out as a possibility at the beginning of the proof of this claim.

\qed(Claim 5)

\medskip

To conclude, we note that if $C$ is incompressible, then conclusion (1) holds; if $C$ is compressible into $H^{out}$, then, by Claim 4, conclusion (3) holds; and if $C$ is compressible into $H^{in}$ but incompressible into $H^{out}$, then, by Claim 3 and Claim 5, at least one of conclusions (1), (2) or (3) holds. \end{proof}

\begin{proof}[Proof of Corollary~\ref{coro:main}]
Let $L$ be a 3-bridge knot with a 6-punctured bridge sphere $\Sigma$ of distance two. Let $C$ be the boundary of a regular neighborhood $H^{in}$ of $F$ as defined at the beginning of the proof of Theorem~\ref{thm:main}. Recall that after isotopying $L$ relative to $\Sigma$ so as to minimize $|C\cap L|$, $|C\cap L|$ is at most $2|F\cap L|-4$ since $F$ is compressible in $B^{\up}$ and in $B^{\dn}$. Since $\partial F$ is essential in $\Sigma$, $|F\cap L|=$ 2, 3 or 4.

If $|F\cap L|=2$, then $|C\cap L|=0$ and $\Sigma$ is distance zero, a contradiction.

If $|F\cap L|=3$, then $|(C\cap B^{\up})\cap L|=1$ and $|(C\cap B^{\dn})\cap L|=1$. Let $\alpha$ be the strand of $L\cap B^{\up}$ that is disjoint from $H^{in}$ and let $E$ be a bridge disk for $\alpha$. By an outermost arc, innermost loop argument, we can assume that $E$ is disjoint from $C$. The boundary of a regular neighborhood of $E$ in $B^{\up}$ is a compressing disk for $F^c$, this contradicts the incompressibly of $F^c$.

If $|F\cap L|=4$, then $|C\cap L|=$ 0, 2 or 4. In the case that $|C\cap L|=$ 0 or 2, one of $C\cap B^{\up}$ or $C\cap B^{\dn}$ is disjoint from $L$ and $F^c$ is compressible, a contradiction. Hence, $|F\cap L|=4$ and $|C\cap L|=4$.

If $C$ is incompressible, then conclusion (1) of the corollary holds.

If $C$ is compressible into $H^{in}$, then $L\cap H^{in}$ is a rational tangle or at least one strand of $L\cap H^{in}$ is knotted. If $L\cap H^{in}$ is a rational tangle, $L$ can be decomposed as the union of three rational tangles $L\cap H^{in}$, $L\cap (H^{out}\cap B^{\up})$ and $L\cap (H^{out}\cap B^{\dn})$ and conclusion (3) of the corollary holds. If at least one strand of $L\cap H^{in}$ is knotted and $C$ is compressible into $H^{in}$, then $L$ is composite or $L$ is 2-bridge, a contradiction. In the case that $L$ is composite, then the distance of $\Sigma$ is at most one by Theorem 1.6 of \cite{D}, a contradiction. The careful reader will notice that Theorem 1.6 of \cite{D} only implies that $\Sigma$ has a compressing disk on one side that is disjoint from a properly embedded disk on the opposite side such that this disk has essential boundary and meets $L$ in at most one point. However, it is an easy exercise to show that any bridge sphere with this property has distance at most one.

If $C$ is compressible in $H^{out}$, then, by Claim 3 of the proof of Theorem~\ref{thm:main}, $L$ is obtained from a link $L'$ via banding where $L'$ has a bridge sphere with the same bridge number as $\Sigma$, but with distance at most one. Hence, $L'$ is the unknot, a 2-bridge knot, or the connected sum of two 2-bridge knots and conclusion (2) of the corollary holds.
\end{proof}

\section{Questions}

Theorem \ref{thm:main} raises a number of interesting questions.

\subsection{Decompositions of bridge surfaces into surfaces of higher distance}
The first concerns the possibility of an analogy between distance 1 and distance 2 bridge spheres.

As motivation, recall that Casson and Gordon \cite{CG} showed that  the existence of a distance 1 Heegaard surface $\Sigma$ for a closed 3-manifold implies the existence of a closed incompressible surface in the 3-manifold. Scharlemann and Thompson \cite{ST} elaborated on the Casson and Gordon construction to show that, in fact, $\Sigma$ can be decomposed into a collection of incompressible surfaces and Heegaard surfaces (for the complement of the incompressible surfaces) each of distance at least 2. Similarly, Thompson \cite{Th} showed (roughly speaking) that if a knot $K \subset S^3$ has a bridge sphere $\Sigma$ which is distance 1 then there is an incompressible meridional surface in the exterior of the knot. Hayashi and Shimokawa \cite{HS} developed that idea and showed that, in fact, a distance $d \leq 1$ bridge surface can be decomposed into essential meridional surfaces and bridge surfaces, each of distance at least $d+1$, for the complementary tangles.

If we view Theorem \ref{thm:main} as an analogue of the Casson-Gordon and Thompson results,  it is natural to wonder:

\begin{question}\label{Q: dist 2 decomp}
If $\Sigma \subset S^3$ is a bridge sphere of distance 2, does $(S^3, K)$ have a tangle decomposition such that $\Sigma$ can be decomposed into bridge spheres, each of distance at least 3, for the tangles in the decomposition?
\end{question}

The proof of Theorem \ref{thm:main} suggests that the sphere $C$ is a possible starting point for such a tangle decomposition. However, in the cases when $C$ is compressible, the bridge sphere induced by $\Sigma$ on the side of $C$ to which $C$ compresses will likely have distance 0. Thus, much needs to be done to obtain a tangle decomposition answering Question \ref{Q: dist 2 decomp}.

Of course, we can be much more ambitious and ask:

\begin{question}
If $\Sigma$ is a bridge surface (of any genus) for a properly embedded 1--manifold $K$ in a closed 3-manifold $M$ of distance $d$, is there a decomposition of $(M,K)$ into (3-manifold, 1-manifold) submanifolds and a decomposition of $\Sigma$ into bridge surfaces, each of distance at least $d + 1$, for the submanifolds?
\end{question}

Hayashi and Shimokawa's work answers the question in the affirmative in the case when $d \leq 1$.

\subsection{Bridge number and distance}

For a knot $K \subset S^3$ we can define $d(K)$ to be the minimal distance of a minimal bridge sphere for $K$. Letting $b(K)$ denote the bridge number of $K$, the pair $(b(K), d(K))$ is an invariant of $K$. Both Conclusion (3) of Theorem \ref{thm:main} and Corollary \ref{coro:main} suggest the following:

\begin{question}
For each $d \geq 0$, is it possible to classify all distance $d$ bridge surfaces of 3-bridge knots?
\end{question}

\begin{question}
Is there a finite number of basic operations on knots in $S^3$ such that a knot $K$ can be obtained, by one of the basic operations, from a knot $K'$ with $(b(K'), d(K')) < (b(K), d(K))$?
\end{question}

\subsection{Dehn surgery questions}

Theorem \ref{thm:main} was inspired, in part, by our studies \cite{BCJTT, BCJTT2} of the relationship between distance and exceptional surgeries on knots. We showed, in particular, that if a knot $K \subset S^3$ of bridge number at least 3 has a reducing surgery or a toroidal surgery then $d(K) \leq 2$. We, therefore, wonder:

\begin{question}
Given a distance 2 bridge sphere $\Sigma$ for a knot $K \subset S^3$, how can we determine if $K$ is a satellite knot?
\end{question}

 \end{document}